\documentclass[a4paper]{article}

\usepackage[english]{babel}
\usepackage[utf8]{inputenc}
\usepackage{amsmath,amsthm,amssymb}
\usepackage{graphicx}
\usepackage[colorinlistoftodos]{todonotes}
\usepackage{enumitem}
\usepackage[all]{xy}
\usepackage[ruled,linesnumbered]{algorithm2e}

\theoremstyle{plain}
\newtheorem{theorem}{\bf Theorem}[section]
\newtheorem{proposition}[theorem]{\bf Proposition}
\newtheorem{lemma}[theorem]{\bf Lemma}

\theoremstyle{definition}
\newtheorem{example}[theorem]{\bf Example}

\newtheorem{remark}[theorem]{\bf Remark}

\newtheorem{problem}[theorem]{\bf Problem}

\DeclareMathOperator{\supp}{supp}

\DeclareMathOperator{\Aut}{Aut}
\DeclareMathOperator{\Orb}{Orb}

\newcommand{\bdot}{\boldsymbol{\cdot}}

\usepackage{authblk}

\usepackage{url}

\title{The Noether numbers and the Davenport constants of the groups of order less than 32}
\author[1]{K\'alm\'an Cziszter 
\thanks{Email: \texttt{cziszter.kalman@gmail.com}\\
Partially supported by  National Research, Development and Innovation Office, NKFIH   grants PD113138, ERC~HU~15 118286 and K115799.}
}
\author[1]{M\'aty\'as Domokos 
\thanks{Email:  \texttt{domokos.matyas@renyi.mta.hu}\\
Supported by National Research, Development and Innovation Office,  NKFIH K 119934.}}
\author[2,1]{Istv\'an Sz\" oll\H osi 
\thanks{Email: \texttt{szollosi@gmail.com}\\
Supported by ERC-AdG 321104 and GTC-31816 (Babe\c s-Bolyai University grant).}}
\affil[1]{MTA R\'enyi Institute,  1053 Budapest, Re\'altanoda utca 13-15, Hungary}
\affil[2]{Faculty of Mathematics and Computer Science, Babe\c s-Bolyai University, str. M. Kog\u alniceanu, nr. 1, 400084, Cluj-Napoca, Romania}
\date{}

\begin{document}
\maketitle

\begin{abstract}
The computation of the Noether numbers of all groups of order less than thirty-two is completed. 
It turns out that for these groups in non-modular characteristic the Noether number is attained on a multiplicity free representation, 
it is strictly monotone on subgroups and factor groups, 
and it does not depend on the characteristic.  Algorithms are developed  and used  to determine 
the small and large Davenport constants of these groups.  For each of these groups the Noether number is greater than the small Davenport constant,  whereas the first example of a group whose Noether number exceeds the large Davenport constant is found, answering partially a question posed by Geroldinger and Grynkiewicz.


\end{abstract}

\section{Introduction} \label{sec:intro}

\subsection{The Noether number} 
Fix a base field $\mathbb{F}$ and a finite group $G$. Given a 
{\it $G$-module} $V$ 
(i.e. a finite dimensional $\mathbb{F}$-vector space $V$ together with an action of $G$ via linear transformations) there is an induced action of $G$ on the {\it symmetric tensor algebra}  
$S(V)$  by $\mathbb{F}$-algebra automorphisms. More concretely, $S(V)$ can be identified with the polynomial algebra $\mathbb{F}[x_1,\dots,x_n]$ where $x_1,\dots,x_n$ is a basis of $V$, on which $G$ acts via linear substitutions of the variables. Noether \cite{noether:1926} proved  that the algebra 
\[S(V)^G=\{f\in S(V) : g\cdot f=f \text{ for all } g\in G\}\] 
of {\it polynomial invariants} is generated by finitely many homogeneous elements. Denote by $\beta(S(V)^G)$ the minimal non-negative integer $d$ such that $S(V)^G$ is generated by its homogeneous components of degree at most $d$. 
Here we refer to the standard grading on $S(V)$, so the variables $x_i$ all have degree one. 
The {\it Noether number} of  $G$ is 
\[\beta^{\mathbb{F}}(G)=\sup \{\beta(S(V)^G) \colon V\mbox{ is a }G\mbox{-module}\}.\]   
The following general facts are well known: 
\begin{equation}\label{eq:noether}
\beta^{\mathbb{F}}(G)\begin{cases}=\infty &\mbox{ when }\mathrm{char}(\mathbb{F})\mid |G| \\
\le |G| &\mbox{ when }\mathrm{char}(\mathbb{F})\nmid |G|;\end{cases} \end{equation} 
(see \cite{richman} for the case $\mathrm{char}(\mathbb{F})\mid |G|$, 
\cite{noether} for the case $\mathrm{char}(\mathbb{F})=0$ and 
\cite{fleischmann}, \cite{fogarty} for the case $0<\mathrm{char}(\mathbb{F})\nmid |G|$). 

From now on we assume that $\mathbb{F}$ is a fixed base field with $\mathrm{char}(\mathbb{F})\nmid |G|$, and write $\beta(G):=\beta^{\mathbb{F}}(G)$ by suppressing from the notation the dependence 
of the Noether number on $\mathbb{F}$. It is well known that the Noether number is unchanged 
when we extend the base field (see Subsection~\ref{subsec:char} for more information), so we may assume in proofs that $\mathbb{F}$ is algebraically closed. 

The exact value of the Noether number is known only for a very limited class of groups. First of all, we have 
\begin{equation}\label{eq:beta=D}
\beta(G)=\mathsf{D}(G)\mbox{ for abelian }G
\end{equation} 
where $\mathsf{D}(G)$ is the Davenport constant (the maximal length of an irreducible zero-sum sequence over $G$); this observation was used first in \cite{schmid}. The exact value of $\mathsf{D}(G)$ is known among others for abelian 
$p$-groups and for abelian groups of rank at most two. Considering non-abelian groups, the Noether number of the dihedral groups was determined in \cite{schmid} (and in   \cite{sezer} for non-modular positive characteristic) along with the Noether numbers of the quaternion group of order $8$ and the alternating group $A_4$. 
Recent works of the first two authors of the present paper (see \cite{CzD:1}, \cite{CzD:2}, 
\cite{Cz1}, \cite{Cz2}, \cite{CzDG}) added a few  more (series) of groups to this short list. 
These results indicated that a complete table of the Noether numbers of ``small'' groups might be within reach. It turned out that indeed, the reduction lemmas from  \cite{CzD:1} and considerations similar to the methods used in the above mentioned papers are sufficient 
to determine the Noether numbers for all groups of order less than $32$. 
Note that the number of non-abelian groups of order $32$ is $44$.  This explains our choice of limiting the scope of this paper to the groups of order less than $32$.

\subsection{The Davenport constants} 
Equality \eqref{eq:beta=D} inspired Geroldinger and 
Grynkiewicz \cite{GeGryn} to look for an analogue in the case of non-abelian groups. 
By a {\it sequence} over the finite group $G$ we mean a finite sequence of elements from $G$ which is unordered and
repetition of terms is allowed. 
A sequence is  {\it product-one} if the product of its elements in an appropriate order is 
$1_G$. A sequence is {\it product-one free} if it has no product-one subsequences. 
The {\it small Davenport constant} $\mathsf{d}(G)$ was defined in \cite{olson-white} as the maximal length of a product-one free sequence. 
A sequence is considered as an element of the free abelian monoid $\mathcal F(G)$ and product-one sequences form a submonoid $\mathcal{B}(G)$  of $\mathcal{F}(G)$. The {\it large Davenport constant} $\mathsf{D}(G)$ was defined in \cite{GeGryn} as the maximal length of an atom (irreducible element) in $\mathcal{B}(G)$. We have the inequality 
\[\mathsf{d}(G)+1\le \mathsf{D}(G)\] 
with equaliy for abelian $G$. The question whether $\beta(G)$ is always between 
$\mathsf{d}(G)+1$ and $\mathsf{D}(G)$ was raised in \cite{GeGryn} (the possible relation between the Noether number and Davenport constants is discussed further in  
\cite{CzDG}). 
Using the implementation of our algorithms presented in Section~\ref{sec:algorithms} we completed the determination of  $\mathsf{d}(G)$ and $\mathsf{D}(G)$ for groups of order less than $32$. It turned out that 
$\mathsf{d}(G)+1\le\beta(G)\le \mathsf{D}(G)$ holds for these groups with the only exception being  the Heisenberg group $H_{27}$ of order $27$, for which we have $\beta(H_{27})>\mathsf{D}(H_{27})$. 

\subsection{Outline of the paper} 

In Section~\ref{sec:table} we give a table containing  the values of the  Noether number and the Davenport constants for each non-abelian group of order less than $32$.  
In Section~\ref{sec:noether} we provide references and proofs to verify 
the Noether numbers in the table. 
We draw consequences from the obtained data and state some open questions in 
Section~\ref{sec:observations}.  
In particular, in  non-modular characteristic for each group of order less than $32$ the Noether number is  attained on a multiplicity free representation (see  Theorem~\ref{thm:multfree}), the Noether number is strictly monotone with respect to taking subgroups or factor groups (see Theorem~\ref{thm:monotone}, which is generalized to arbitrary finite groups in the subsequent paper \cite{CzD:4}) and the Noether number does not depend on the characteristic (see Theorem~\ref{thm:char-indep}).  
In Section~\ref{sec:davenport} we turn to the Davenport constants. 
Notation and known results are recalled in Sections~\ref{sec:monoid} and ~\ref{subsec:known-davenport}. 
In Section~\ref{sec:heisenberg} we present a theoretical proof for the fact that $\mathsf{D}(H_{27})=8$; this seems to be of special interest because so far this is the only known example of a group for which the Noether number is greater than the large Davenport constant.   Section~\ref{sec:algorithms} contains the description of the algorithms we employed to compute  the  Davenport constants given in the table in Section~\ref{sec:table}. 

\section{Noether numbers and Davenport constants} \label{sec:table}

The classification of all groups of order less than $32$ is given e.g. in \cite[Chapter~22]{humphreys}. 
In the table below we present the Noether numbers and Davenport constants of all the non-abelian groups of order less than $32$.  In the first column we also give for reference the GAP (see \cite{GAP4}) identification numbers $(m,n)$ using which these groups can be constructed in GAP by   the function call \texttt{SmallGroup(m,n)}.

\[ 
\begin{array}{c|c|c|c|c|c}
GAP &G  & \mathsf d & \beta & \mathsf D &\text{reference for $\beta$}\\ \hline
(6,1) & S_3=Dih_6 & 3&4& 6 & \text{\cite[Theorem~10.3]{CzD:2}} \\
(8,3) & Dih_8 & 4&5& 6 & \text{\cite[Theorem~10.3]{CzD:2}} \\
(8,4)& Q_8=Dic_8  & 4 & 6 & 6  & \text{\cite[Theorem~10.3]{CzD:2}} \\
(10,1) & Dih_{10} & 5&6& 10 & \text{\cite[Theorem~10.3]{CzD:2}} \\
(12,1) & Dic_{12} = C_3 \rtimes C_4 & 6 & 8 & 9 & \text{\cite[Theorem~10.3]{CzD:2}}\\
(12,3)&A_4 & 4 & 6 & 7 & \text{\cite[Theorem 3.4]{CzD:1}}\\
(12,4) & Dih_{12} & 6&7& 9 & \text{\cite[Theorem~10.3]{CzD:2}} \\
(14,1) & Dih_{14} & 7&8& 14 & \text{\cite[Theorem~10.3]{CzD:2}} \\
(16,3) & (C_2\times C_2) \rtimes C_4 =  (C_4 \times C_2) \rtimes_{\psi} C_2 & 5 & 6 & 7 & \text{Proposition~\ref{prop:k4_rtimes_c2}}\\
(16,4) & C_4 \rtimes C_4 & 6 &7 & 8 & \text{Proposition~\ref{prop:C4C4}}\\
(16,6) & M_{16} & 8&9 &10 & \text{\cite[Theorem~10.3]{CzD:2}}\\
(16,7) & Dih_{16} & 8&9 &12 & \text{\cite[Theorem~10.3]{CzD:2}}\\
(16,8) & SD_{16}  & 8& 9 & 12 & \text{\cite[Theorem~10.3]{CzD:2}}\\
(16,9) & Dic_{16}  & 8& 10 & 12 & \text{\cite[Theorem~10.3]{CzD:2}}\\
(16,11)& Dih_8 \times C_2 \; = \; (C_4 \times C_2) \rtimes_{-1} C_2 & 5 & 6 & 7 
& \text{\cite[Corollary 5.5]{CzD:2}} \\
(16,12) & Q_8 \times C_2 & 5 & 7 & 7 & \text{Proposition~\ref{prop:Q8C2}}\\
(16,13)& (Pauli) \; = \; (C_4 \times C_2) \rtimes_{\phi} C_2  & 5 & 7  & 7 & \text{\cite[Example~5.4]{CzDG}} \\
(18,1) & Dih_{18} & 9&10& 18 & \text{\cite[Theorem~10.3]{CzD:2}} \\
(18,3) &S_3 \times C_3 & 7 & 8 & 10 &\text{Proposition~\ref{prop:S3C3}}\\
(18,4) & (C_3\times C_3) \rtimes_{-1} C_2 & 5 & 6 & 10 &  \text{\cite[Corollary 5.5]{CzD:2}}\\
(20,1) & Dic_{20}  & 10 & 12 & 15 & \text{\cite[Theorem~10.3]{CzD:2}}\\
(20,3)& C_5 \rtimes C_4 & 7 & 8 & 10 &\text{\cite[Proposition 3.2]{CzD:1}}\\
(20,4) & Dih_{20} & 10&11& 15 & \text{\cite[Theorem~10.3]{CzD:2}} \\
(21,1)& C_7 \rtimes C_3 & 8 & 9 & 14 & \text{\cite[Proposition 2.24]{CzD:1}}\\
(22,1) & Dih_{22} & 11&12& 22 & \text{\cite[Theorem~10.3]{CzD:2}} \\
(24,1) & C_3 \rtimes C_8 
 & 12 & 13 & 15 & \text{\cite[Theorem~10.3]{CzD:2}}\\
(24,3)& SL_2(\mathbb{F}_3) = \tilde{A}_4& 7 & 12 & 13 & \text{\cite[Corollary 3.6]{CzD:1}} \\
(24,4) & Dic_{24} = C_3 \rtimes Q_8  & 12&14 & 18 & \text{\cite[Theorem~10.3]{CzD:2}}\\
(24,5) & Dih_6\times C_4  & 12& 13 & 15 & \text{\cite[Theorem~10.3]{CzD:2}}\\
(24,6) & Dih_{24} & 12&13& 18 & \text{\cite[Theorem~10.3]{CzD:2}} \\
(24,7) & Dic_{12} \times C_2 & 8 &9 & 11 & \text{Proposition~\ref{prop:Dic12C2}}\\ 
(24,8)& C_3 \rtimes Dih_8 = (C_6 \times C_2) \rtimes_{\gamma} C_2 & 7 & 9 & 14 & \text{Proposition~\ref{prop:C3Dih8}}\\
(24,10)& Dih_{8} \times C_3 & 12&13 & 14 & \text{\cite[Theorem~10.3]{CzD:2}}\\
(24,11)& Q_8 \times C_3 & 12&13 & 14 & \text{\cite[Theorem~10.3]{CzD:2}}\\
(24,12)& S_4 & 6 & 9  & 12 & \text{\cite[Example~5.3]{CzDG}}\\
(24,13) & A_4 \times C_2 & 7 & 8 & 10&\text{Proposition~\ref{prop:a4xc2}}\\
(24,14) & Dih_{12} \times C_2 = (C_ 6 \times C_2)\rtimes_{-1} C_2& 7 &8  & 10 
&  \text{\cite[Corollary 5.5]{CzD:2}}\\
(26,1) & Dih_{26}  & 13& 14 & 26 & \text{\cite[Theorem~10.3]{CzD:2}}\\
(27,3)& H_{27}=UT_3(\mathbb{F}_3) & 6 & 9
& {8} & \text{\cite[Corollary~15]{Cz2}}\\
(27,4)& M_{27} = C_9 \rtimes C_3 & 10 & 11 & 12 & \text{\cite[Remark~10.4]{CzD:2}}\\
(28,1) & Dic_{28}=C_7\rtimes C_4  & 14 & 16 & 21 & \text{\cite[Theorem~10.3]{CzD:2}}\\ 
(28,3) & Dih_{28} & 14& 15 & 21 & \text{\cite[Theorem~10.3]{CzD:2}}\\
(30,1) & Dih_6\times C_5  & 15&  16 & 18 & \text{\cite[Theorem~10.3]{CzD:2}}\\
(30,2) & Dih_{10}\times C_3 & 15& 16 & 20&  \text{\cite[Theorem~10.3]{CzD:2}}\\ 
(30,3) & Dih_{30}  & 15& 16 & 30 & \text{\cite[Theorem~10.3]{CzD:2}}\\
\end{array}
\] 

In this table $S_3$ and  $S_4$ are the symmetric groups of degree $3$ and $4$, $Q_8$ is the quaternion group of order $8$, $A_4$ is the alternating group of degree $4$  and $\tilde A_4$ is the binary tetrahedral group, 
$H_{27}$ is the Heisenberg group of order $27$ 
(i.e. the group of unitriangular $3\times 3$ matrices over the $3$-element field). 
For $m\ge 2$, 
$Dih_{2m}$ is the dihedral group of order $2m$, 
$Dic_{4m}$ is the {\it dicyclic group} given by generators and relations 
\[Dic_{4m} := \langle a,b\mid a^{2m}=1, \quad  b^2=a^{m}, \quad bab^{-1}= a^{-1}\rangle.\]
Using  for the semidirect product of two cyclic groups the notation 
\[ C_m \rtimes_d C_n = \langle a,b\mid    a^m=1, b^n=1, bab^{-1}=a^d \rangle  \quad \text{ where } d \in \mathbb{N}\mbox{ is coprime to }m \]
we have that 
$SD_{2^k}$ is the {\it semidihedral group} 
\[SD_{2^k} = C_{2^{k-1}} \rtimes_d C_2, \qquad d=2^{k-2}-1 \quad (k\geq 4),\]
and for  a prime $p$ and $k\ge 3$, 
\[M_{p^k} = C_{p^{k-1}} \rtimes_d  C_p, \qquad d={p^{k-2}+1}.\]
Moreover, the symbol $\rtimes$ always stands for a semidirect product that is not a direct product.

\section{Noether numbers}\label{sec:noether}

\subsection{Abelian groups}\label{subsec:abelian}
It has been long known that for an  abelian group $G$ we have $\mathsf d(G)+1=\beta(G)= \mathsf D(G)$, see \cite{CzDG} for a recent survey largely motivated by this fact. 
Therefore we can restate known results on the Davenport constants 
of abelian groups in terms of the Noether number: 

\begin{itemize}
\item if $G$ is cyclic then $\beta(G) = |G|$ (see for example \cite{schmid});
\item if $G$ is of rank two, i.e. $G = C_n \times C_m$ for some $m \mid n$ then $\beta(G) = n+m-1$ 
(see for example  \cite[Theorem 5.8.3]{Ge-HK06a}); 
\item if $G$ is a $p$-group, i.e. $G = C_{p^{n_1}}\times \ldots \times C_{p^{n_r}}$ then $\beta(G) = 1+ \sum_{i=1}^r (p^{n_i}-1) $ 
(see for example  \cite[Theorem 5.5.9]{Ge-HK06a});   
\item if $G= C_2 \times C_2 \times C_{2n}$ then $\mathsf D(G) = 2n+2$ by \cite{emde}.
\end{itemize}
All abelian groups of order less than $32$ fall under one of the four cases above.  
We note that more recent  progress on the Davenport constants of abelian groups can be found in  \cite{bhowmik-puchta}, 
\cite[Corollary 4.2.13]{geroldinger}, 
\cite{wschmid}, \cite{chen-savchev}. 

\subsection{Groups with a cyclic subgroup of index two}\label{subsec:indextwo}
Let $G$ be a non-cyclic group having a cyclic subgroup of index two. 
According to  \cite[Theorem~10.3]{CzD:2} we have  
\begin{align} \label{indextwo}
\beta(G) = \frac{1}{2} |G|  +
\begin{cases}
2 	& \text{ if } G=Dic_{4m}, \text{  $m>1$};\\
1	& \text{ otherwise. }
\end{cases}
\end{align}
Formula \eqref{indextwo} yields the Noether number for $27$ groups out of the $45$ groups  from the table in Section~\ref{sec:table}. 

\subsection {The generalized dihedral groups}\label{subsec:generalized dihedral} 

The semi-direct product $Dih(A) := A \rtimes_{-1} C_2$, where $A$ is an abelian group on which $C_2$ acts by  inversion, is called the \emph{generalized dihedral group} obtained from $A$. According to \cite[Corollary~5.5]{CzD:2} we have 
\begin{align}\label{gendih}\beta(Dih(A)) = \mathsf D(A) +1.\end{align}
Combining this with the known values of Davenport constants given in Section~\ref{subsec:abelian} we can compute the Noether number of the three generalized dihedral groups of order less than $32$ which are not themselves dihedral groups (these are 
$(C_4 \times C_2)\rtimes_{-1} C_2 =Dih_8 \times C_2$, 
$(C_3\times C_3) \rtimes_{-1} C_2$, 
$(C_ 6 \times C_2)\rtimes_{-1} C_2 = Dih_{12} \times C_2$). 

\subsection{Cases when the reduction lemmas give  exact results}

The $k$th \emph{Noether number} $\beta_k(S(V)^G)$ (where $k$ is a positive integer) was defined  in \cite[Section 1.2]{CzD:1} 
as the top degree of the factor space $S(V)^G/(S(V)^G_+)^{k+1}$, where 
$S(V)^G_+$ stands for the sum of the positive degree homogeneous components of $S(V)^G$. The supremum of $\beta_k(S(V))^G$ as $V$ ranges over all $G$-modules over $\mathbb{F}$ is denoted by $\beta_k(G)$ (for an abelian group $G$, $\beta_k(G)$ equals the $k$th Davenport constant $\mathsf{D}_k(G)$ introduced in \cite{halter-koch}). 
In the special case $k=1$ we have $\beta_1(S(V)^G)=\beta(S(V)^G)$ and hence 
$\beta_1(G)=\beta(G)$. 
Using the $k$th Noether number one can get upper bounds on $\beta(G)$ by the following reduction lemma:  
\begin{align}
\label{red}
\beta(G) &\le \beta_{\beta(G/N)}(N) && \text{ for }N \triangleleft G \mbox{ by  \cite[Lemma 1.4]{CzD:1}}. 
\end{align} 
Lower bounds on $\beta(G)$ can be derived from the following inequality: 
\begin{align} 
\label{also1}
\beta(G) &\ge \beta(G/N) + \beta(N) -1 && \text{ if $G/N$ is abelian by \cite[Theorem 4.3]{CzD:2}}. 
\end{align}
These results already suffice to establish the precise value of the Noether number for several groups considered below when we combine them with the following formula for the $k$th Noether number:
\begin{align}
\label{halterkoch}
\beta_k(C_n \times C_m) = nk+m-1 & \text{ for $m \mid n$ by \cite[Proposition 5]{halter-koch}}. 
\end{align}

\begin{proposition}\label{prop:C4C4}
$\beta(C_4 \rtimes C_4) =7$.
\end{proposition}
\begin{proof}
We have the lower bound $\beta(C_4 \rtimes C_4) \ge 2 \beta(C_4)-1 =7$ by \eqref{also1}. On the other hand $G = C_4 \rtimes C_4$ contains a normal subgroup isomorphic to the Klein four-group $K_4=C_2 \times C_2$ such that $G/K_4 \cong K_4$. Hence by \eqref{red} and \eqref{halterkoch} we have 
$\beta(G) \le \beta_{\beta(K_4)}(K_4)=7$. 
\end{proof}

\begin{proposition}\label{prop:Q8C2}
$\beta(Q_8 \times C_2) =7$.
\end{proposition}
\begin{proof}
We have the lower bound $\beta(Q_8 \times C_2) \ge \beta(Q_8) +\beta(C_2)-1 = 7$ by \eqref{also1} since $\beta(Q_8) = 6$ by \cite[Lemma~10.1]{schmid}. On the other hand $G$ has a normal subgroup $K_4$ such that $G/K_4 \cong K_4$. Hence again $\beta(G) \le \beta_{\beta(K_4)}(K_4)=7$ by \eqref{red} and \eqref{halterkoch}.
\end{proof}

\begin{proposition} \label{prop:S3C3}
$\beta(S_3 \times C_3) =8$.
\end{proposition}

\begin{proof}
We have  $\beta(S_3 \times C_3)=\beta(C_3\rtimes_{-1}C_6) \ge
 \beta(C_3)+ \beta(C_6)  -1 = 8$ by \eqref{also1} 
and the upper bound $\beta(S_3 \times C_3) \le \beta_{2}(C_3 \times C_3)=8$ by \eqref{red} and \eqref{halterkoch}.
\end{proof}

\begin{proposition} \label{prop:Dic12C2}
$\beta(Dic_{12} \times C_2) =9$.
\end{proposition}
\begin{proof}
We have $\beta(Dic_{12} \times C_2) \ge \beta(Dic_{12}) +\beta(C_2)-1 = 8+2-1 =9$ by \eqref{also1} and \eqref{indextwo}.
On the other hand the center of the group $G =Dic_{12} \times C_2$ is isomorphic to the Klein four-group $K_4 = C_2 \times C_2$ and we have $G/K_4 \cong Dih_6$. So we get $\beta(G) \le \beta_{\beta(Dih_6)}(K_4) = \beta_4(K_4) = 2 \cdot 4 +1 =9$ by \eqref{indextwo}, \eqref{red} and \eqref{halterkoch}.
\end{proof}

\begin{proposition} \label{prop:C3Dih8}
$\beta(C_3 \rtimes Dih_8) = 9$.
\end{proposition}
\begin{proof} 
The group $G$ has an index two subgroup $N\cong C_3\rtimes C_4=Dic_{12}$. 
By \eqref{indextwo} we have 
$\beta(N)=\frac 12|N|+2=8$, implying by \eqref{also1} that 
$\beta(G)\ge 8+2-1=9$. 

For the reverse inequality observe that the  kernel of the action of $Dih_8$ on $C_3$ is isomorphic to the Klein group $K_4= C_2 \times C_2$. So the subgroup $K_4 \le Dih_8$ is normal in $G$ and as $Dih_8 /K_4 \cong C_2$ acts by inversion on $C_3$ we have $G/K_4 \cong Dih_6$. Therefore $\beta(G) \le \beta_{\beta(Dih_6)}(K_4) = \beta_4(K_4) = 2 \cdot 4 +1 =9$  by \eqref{indextwo}, \eqref{red} and \eqref{halterkoch}.
\end{proof}

\begin{example}\label{ex:c3dih8}
For later reference let us  construct  here a representation of $G = C_3 \rtimes Dih_8$ (where the kernel of the conjugation action of $Dih_8$ on $C_3$ is isomorphic to $C_2\times C_2$) on which the Noether number is  attained.
Consider the two-dimensional $G$-module $U=\mathbb{F}^2$ on which the representation is given by the matrices
\[
a= \begin{pmatrix}
0 & 1  \\
1 & 0   
\end{pmatrix}, \quad 
b=\begin{pmatrix}
-1 & 0  \\
0 & 1  
\end{pmatrix}, \quad
c=\begin{pmatrix}
\zeta & 0  \\
0 & \zeta^2 
\end{pmatrix} 
\]
where $\zeta$ is a primitive third root of unity. 
Here $\langle a,b \rangle\cong Dih_8$,  $c^a = c^{-1}$ and $c^b=c$, so these matrices are indeed generating the group in question. Since this group is generated also by the pseudo-reflections $a,b,$ and $ac$, the $G$-module structure of $S(U)$ is well known by the Shephard-Todd-Chevalley theorem \cite{chevalley}, \cite{shephard-todd}.  In particular, $S(U)^G=\mathbb{F}[x,y]^G$ is a polynomial ring generated by $(xy)^2,x^6+y^6$. 
The element $x^7y-xy^7$ is not contained in the ideal of $S(U)$ generated by the above two invariants. It  spans a one-dimensional $G$-invariant subspace isomorphic to the one-dimensional $G$-module $\mathbb{F}_{\chi}$ corresponding to the order two homomorphism $\chi:G\to \mathbb{F}^{\times}$ given by the determinant on $GL(U)$.  
It follows that in $S(U\oplus \mathbb{F}_{\chi})$ the element $(x^7y-xy^7)z$ is an indecomposable $G$-invariant, where $z$ spans the summand $\mathbb{F}_{\chi}$. 
\end{example}

\subsection{Some groups with an abelian normal subgroup of index two or three} 

In this section we shall discuss a group $G$ with an abelian normal subgroup $A$ of index $2$ or $3$. Let $V$ be a $G$-module over an algebraically closed base field $\mathbb{F}$. We shall fix the following notation:  
$I:=S(V)^A$, $R:=S(V)^G$. By Clifford theory we know that each irreducible $G$-module is $1$-dimensional or is induced from a $1$-dimensional $A$-module. 
Therefore it is possible to choose the variables in $S(V)$ to be $A$-eigenvectors that are permuted up to scalars by the action of $G$. We shall tacitly assume that the variables were chosen that way. It follows that $I$ is spanned by $A$-invariant monomials, permuted by $G$ up to scalars. 
It will be convenient to use the notation $f^g=g^{-1}\cdot f$ for 
$g\in G$ and $f\in S(V)$. 
For a non-zero scalar multiple $x$ of a variable let us denote by  $\theta(x)\in \widehat A$ its weight (i.e. $x^a=(\theta(x)(a))x\in \mathbb{F} x$ for all $a\in A$). 
Here $\widehat A$ is the character group of $A$, i.e. the abelian group consisting of the group homomorphisms $A\to \mathbb{F}^\times$. Note that $\widehat A\cong A$ when $\mathbb{F}$ is algebraically closed. 
For a non-zero scalar multiple $w$ of a monomial $y_1\dots y_d$ the multiset  
$\Phi(w)=
\{\theta(y_1),\dots,\theta(y_d)\}$ is called the {\it weight sequence} of $w$. It is a sequence over the abelian group 
$\widehat A$ (where the order of the elements in a {\it sequence} is disregarded and repetition is allowed). Sequences over $\widehat A$ form a monoid with multiplication denoted by ``$\bdot$''  
(see Section~\ref{sec:davenport} for more discussion of this monoid), such that for monomials 
$w,w'$ we have $\Phi(ww')=\Phi(w)\bdot \Phi(w')$. 
The action of $G$ on $A$ by conjugation induces an action on $\widehat A$ by $\theta^g(a)=\theta(gag^{-1})$.  For a variable $x$, $g\in G$ and $a\in A$ we have that $x^g$ is a non-zero scalar multiple of a variable, and 
$(x^g)^a=(x^{gag^{-1}})^g=(\theta(x)(gag^{-1})x)^g=\theta(x)^g(a)x^g$, showing the equality  $\theta(x^g)=\theta(x)^g$.  Consequently,  for a non-zero scalar multiple $w$ of a monomial we have $\Phi(w^g)=\Phi(w)^g$ where the action of $G$ on $\widehat A$ is extended componentwise to sequences over $\widehat A$. 
Denote by $\tau=\tau^G_A:I\to R$ the {\it relative transfer map} 
$f\mapsto \sum_{g\in G/A} f^g$. It is a surjective $R$-module homomorphism, preserving the grading (inherited from the standard grading on the polynomial ring $S(V)$). 
The following lemma provides the common basis for the proofs of the upper bounds for Noether numbers below:  

\begin{lemma}\label{lemma:s1s2} Suppose that for any zero-sum sequence 
$S$ over $\widehat A$ with $|S|>d$ which does not factor as the product  of  
$1+[G:A]$ non-empty zero-sum sequences, we have a factorization 
$S=S_1\bdot S_2$ as the product of two zero-sum sequences satisfying the following: 
for each $g\in G\setminus A$, the sequence $S_1\bdot S_2^g$ factors as the product of $1+[G:A]$ non-empty zero-sum sequences. 
Then we have the inequality $\beta(G)\le d$.  
\end{lemma} 
\begin{proof} 
By Proposition 1.5 in \cite{CzD:1}, in order to prove $\beta(S(V)^G)\le d$ it is sufficient to show that $I_{\ge d+1}\subseteq I_+R_+$ where $I_{\ge d+1}$ stands for the sum of homogeneous components of $I$ of degree at least $d+1$. That is, we need to show that any $A$-invariant monomial $w$ with $\deg(w)> d$ belongs to $I_+R_+$. 
So take an arbitrary $A$-invariant monomial with $\deg(w)>d$, and denote by $S$ its weight sequence $\Phi(w)$. If $S$ factors into a product of $1+[G:A]$ zero-sum sequences, then 
$w$ factors as the product of $1+[G:A]$ non-trivial $A$-invariant monomials. 
Thus $w\in  (I_+)^{1+[G:A]}$. On the other hand  by Proposition 1.6 in \cite{CzD:1} we have 
$(I_+)^{1+[G:A]}\subseteq I_+R_+$, implying in turn that $w\in I_+R_+$.  
Next, suppose that $S$ does not factor as the product of $1+[G:A]$ non-empty zero-sum sequences. Then by assumption we have a factorization 
$S=S_1\bdot S_2$ with the properties in the statement. The monomial $w$ factorizes as 
$w=w_1w_2$ with $\Phi(w_1)=S_1$, $\Phi(w_2)=S_2$. 
We have the equality 
\[w=w_1\tau(w_2)-\sum w_1w_2^g,\] 
where the summation above ranges over a set of representatives of the $A$-cosets in $G$ which are different from $A$. For each such summand, $\Phi(w_1w_2^g)=S_1\bdot S_2^g$ factors as the product of $1+[G:A]$ non-empty zero-sum sequences, so $w_1w_2^g\in (I_+)^{1+[G:A]}\subseteq I_+R_+$. The first summand $w_1\tau(w_2)$ above is also contained in $I_+R_+$ since $\tau(w_2)\in R_+$, implying in turn that $w\in I_+R_+$. 
Thus we showed $I_{\ge d+1}\subseteq I_+R_+$. 
\end{proof}

\begin{proposition}\label{prop:a4xc2} 
$\beta(A_4\times C_2)= 8$.
\end{proposition}
\begin{proof}
First we show the inequality $\beta(A_4\times C_2)\le 8$. 
We have $A_4=K\rtimes \langle g\rangle$ where $K\cong C_2\times C_2$ and $\langle g\rangle \cong C_3$, so 
$G:=A_4\times C_2$ contains an abelian normal subgroup $A=K\times C_2$, where $g$ centralizes the last summand $C_2$, and conjugation by $g$ gives an order $3$ automorphism of $K$. Write the character group $\widehat A$ additively, and denote its elements by 
$\{(0,\varepsilon),(a,\varepsilon),(b,\varepsilon), (c,\varepsilon)\mid \varepsilon =0,1\}$, where $(0,0)$ is the zero element, $\{(a,0),(b,0),(c,0),(0,0)\}$ is a subgroup, and the action of $\langle g\rangle$  on $A$ induces the action on $\widehat A$ given by $(a,\varepsilon)^g=(b,\varepsilon)$, $(b,\varepsilon)^g=(c,\varepsilon)$, 
$(c,\varepsilon)^g=(a,\varepsilon)$.  
We shall apply Lemma~\ref{lemma:s1s2}, so take a zero-sum sequence $S$ over $\widehat A$ with $|S|\ge 9$, such that $S$ does not factor as the product of four non-empty zero-sum sequences over $\widehat A$.  
Since $\beta_2(A)=7$ by Lemma 3.7 in \cite{delorme}, $0$ does not occur in $S$ 
(otherwise $S=\{0\}\bdot T$ and $|T|\ge 8>\beta_2(A)=7$ implies that $T$ is the product of three non-empty zero-sum sequences, a contradiction). 
On the other hand, $|S|>7=|\widehat A\setminus \{(0,0)\}|$ implies that $S$ contains an element $s$ with multiplicity at least $2$, hence $S=T_0 \bdot T$, where $T_0,T$ are zero-sum sequences, $T_0=\{s,s\}$ has length $2$, so $|T|\ge 7$. 
Note that $T$ does not contain a zero-sum subsequence of length $2$ (since otherwise 
$\mathsf D(A)=4$ would imply that $S$ is the product of four non-empty zero-sum sequences). 
It follows that $T$ consists of the non-zero elements in $\widehat A$ (each having multiplicity one), and so 
$T=T_1T_2$ where $T_1=\{(a,1),(b,0),(c,0),(0,1)\}$, $T_2=\{(a,0),(b,1),(c,1)\}$. 
Set $S_1:=T_0\bdot T_1$ and $S_2:=T_2$. The factorization $S=S_1\bdot S_2$ fulfills the requirements of Lemma~\ref{lemma:s1s2}: indeed, $(a,1)=(c,1)^g$ occurs both in $S_1$ and $S_2^g$, so $S$ can be written as $S=T_0\bdot \{(a,1),(a,1)\}\bdot U$ where $U$ is a zero-sum sequence of length $5$. Hence by $\mathsf D(A)=4$ we get that $S_1\bdot S_2^g$ factors as the product of four non-empty zero-sum sequences. 
Similarly, $(a,1)=(b,1)^{g^2}$ occurs also in $S_2^{g^2}$, hence $S_1\bdot S_2^{g^2}$ 
also factors as the product of four non-empty zero-sum sequences.
By Lemma~\ref{lemma:s1s2} we conclude that $\beta(G)\le 8$. 

The subgroup $K$ is normal in $G$ and $G/K\cong C_6$, hence 
by \eqref{also1} we obtain the reverse inequality 
$\beta(G)\ge \beta(K)+\beta(C_6)-1=3+6-1=8$. 
\end{proof}  

\begin{example}\label{ex:a4xc2}
For later use we present a concrete  $A_4\times C_2$-module on which the Noether number is attained. We keep the notation used in the proof of 
Proposition~\ref{prop:a4xc2}. 
Let $W$ be the $3$-dimensional irreducible $A_4$-module (the non-trivial direct summand in the $4$-dimensional standard permutation representation of $A_4$) viewed as a $G$-module under the natural surjection $G\to A_4$  with kernel $C_2$. Let $U$ be the non-trivial $1$-dimensional $C_2$-module viewed as a $G$-module under the natural surjection $G\to C_2$ with kernel $A_4$. Set $V:=W\oplus (W\otimes U) \oplus U$. 
Then $S(V)=\mathbb{F}[x_1,x_2,x_3,y_1,y_2,y_3,z]$ where $x_1,x_2,x_3$ are $A$-eigenvectors  with weight 
$(a,0),(b,0),(c,0)$ and they are permuted cyclically by $g$.  Similarly  
$y_1,y_2,y_3$ are $A$-eigenvectors with weight $(a,1),(b,1),(c,1)$ and they are permuted cyclically by $g$. The last variable $z$ is fixed by $g$ and is an $A$-eigenvector with weight $(0,1)$. Denote by $\tau=\tau^G_A:I\to R$ the relative transfer map $f\mapsto  f+f^g+f^{g^2}$. Then $R$ is spanned as an $\mathbb{F}$-vector space by $\tau(w)$ where 
$w$ ranges over the $A$-invariant monomials. There is an $\mathbb{N}^3$-grading on $\mathbb{F}[V]$ given by 
\[\deg_3(x_1^{a_1}x_2^{a_2}x_3^{a_3}y_1^{b_1}y_2^{b_2}y_3^{b_3}z^d):=(a_1+a_2+a_3,b_1+b_2+b_3,d).\] 
This is preserved by the action of $G$, hence $I$ and $R$ are $\mathbb{N}^3$-graded subalgebras. Moreover, $\tau$ preserves the $\mathbb{N}^3$-grading. We claim that $\tau(x_1^3x_2^3y_3z)$ is indecomposable, that is, it is not contained in $(R_+)^2$. Suppose to the contrary that  $\tau(x_1^3x_2^3y_3z)$ is a linear combination of elements of the form $\tau(w)\tau(w')$ where  $w$ and $w'$ are non-trivial $A$-invariant monomials  and $\deg_3(w)+\deg_3(w')=(6,1,1)$. There is no $A$-invariant variable in $\mathbb{F}[V]$, so both $w$ and $w'$ above have total degree at least $2$. 
The $\langle g\rangle$-orbits of the $A$-invariant monomials $w$ with $\deg_3(w)=(*,*,1)$, 
$\deg_3(w)\neq  (6,1,1)$ and where $\deg_3(w)$ is dominated by $(6,1,1)$ are 
\[x_1x_2y_3z,\quad x_1^3x_2y_3z,\quad x_1x_2^3y_3z,\quad x_1x_2x_3^2y_3z.\] 
The $G$-invariants of degree $2$ or $4$ depending only on $x_1,x_2,x_3$ are 
\[\tau(x_1^2)=x_1^2+x_2^2+x_3^2,\quad  \tau(x_1^2x_2^2)=x_1^2x_2^2+x_1^2x_3^2+x_2^2x_3^2, \quad  
\tau(x_1^4)=x_1^4+x_2^4+x_3^4.\] 
It follows that 
\begin{align*}
\tau(x_1^3x_2^3y_3z) &= \tau(x_1x_2y_3z)(\lambda_1 \tau(x_1^4)+\lambda_2 \tau(x_1^2x_2^2)) \\
& + \tau(x_1^2)(\mu_1   \tau(x_1^3x_2y_3z)+\mu_2\tau(x_1x_2^3y_3z)+\mu_3\tau(x_1x_2x_3^2y_3z))
\end{align*} 
for some $\lambda_1,\lambda_2,\mu_1,\mu_2,\mu_3\in \mathbb{F}$.  
Comparing the coefficients of 
$x_1^5x_2y_3z$, $x_1x_2^5y_3z$, $x_1x_2x_3^4y_3z$, $x_1x_2^3x_3^2y_3z$
 we conclude 
 \[0=\lambda_1+\mu_1=\lambda_1+\mu_2=\lambda_1+\mu_3=\lambda_2+\mu_2+\mu_3.\] 
 It follows that the coefficient $\lambda_2+\mu_1+\mu_2$ of $x_1^3x_2^3y_3z$ on the right hand side is $0$, 
 whereas on the left hand side it is $1$. 
 This contradiction implies that $\tau(x_1^3x_2^3y_3z)\notin (R_+)^2$, hence 
 $\beta(S(V)^G)\ge 8$. 
\end{example}

\begin{proposition}\label{prop:k4_rtimes_c2} 
$\beta((C_2\times C_2)\rtimes C_4)= 6$.
\end{proposition}
\begin{proof} 
We have $G=(C_2\times C_2)\rtimes \langle g\rangle$ where $\langle g\rangle \cong C_4$. 
The group $G$ contains the abelian normal subgroup $A:=C_2\times C_2\times \langle g^2\rangle$. 
We use the same notation for the elements of $\widehat A$ as in the proof of Proposition~\ref{prop:a4xc2}. 
The action of $\langle g\rangle$ on $\widehat A$ is given by 
\[(a,\varepsilon)^g=(b,\varepsilon),\quad (b,\varepsilon)^g= (a,\varepsilon), \quad (c,\varepsilon)^g =(c,\varepsilon), 
\quad (0,\varepsilon)^g=(0,\varepsilon)\quad \mbox{ for }\varepsilon=0,1.\] 
Take a zero-sum sequence $S$ over  $\widehat A$ with $|S|\ge 7$ which is not the product of three non-empty zero-sum sequences. Since  $\mathsf D(A)=4$ by Lemma 3.7 in \cite{delorme},  $S$  has no zero-sum subsequence of length at most $2$. It follows that 
$S$ consists of the non-zero elements of $\widehat A$ (each element  having multiplicity $1$),  hence 
$S=S_1\bdot S_2$ where $S_1=\{(b,0),(b,1),(0,1)\}$ and $S_2=\{(a,0),(a,1),(c,0),(c,1)\}$. 
We have that  
\[S_1\bdot S_2^g=\{(b,0),(b,0)\}\bdot \{(b,1),(b,1)\}\bdot \{(c,0),(c,1),(0,1)\}\] 
is the product of $3=1+[G:A]$ zero-sum sequences. 
By Lemma~\ref{lemma:s1s2} we conclude the desired inequality 
$\beta(G)\le 6$.  

On the other hand  
$\beta(G)\ge \beta(C_2\times C_2)+\beta(C_4)-1=3+4-1=6$ by \eqref{also1}.  
\end{proof}

\begin{example}\label{ex:k4_rtimes_c2}
For later use we present a $G$-module on which the 
Noether number of the group $G=(C_2 \times C_2)\rtimes C_4$ is attained. Consider the $G$-module $V=\mathbb{F}^4$ on which the action is given by the matrices 
\[
a= \begin{pmatrix}
1 & 0 & 0 & 0 \\
0 & -1 & 0 & 0 \\
0 & 0 & 1  & 0 \\
0 & 0 & 0 & -1
\end{pmatrix}, \quad 
b=\begin{pmatrix}
-1 & 0  & 0 & 0 \\
0 & 1  & 0 & 0 \\ 
0 & 0 & 1 & 0 \\
0 & 0 & 0 & -1
\end{pmatrix}, \quad
c=\begin{pmatrix}
0 & 1 & 0 & 0  \\
1 & 0 &  0 & 0 \\
0 & 0 & \omega & 0 \\ 
0 & 0 & 0 & \omega
\end{pmatrix} 
\]
where $\omega$ is a primitive fourth  root of unity. 
Then $a,b$ generate a subgroup of $GL(V)$ isomorphic to $C_2\times C_2$, 
$c^4$ is the identity matrix, and conjugation by $c$ interchanges $a$ and $b$. 
So $G$ can be identified with the subgroup of $GL(\mathbb{F}^4)$ generated by $a,b,c$. 
Moreover, the $2\times 2$ upper left blocks of $a,b,c$ give the reflection representation of $\overline G=G/\langle c^2\rangle \cong Dih_8$ on $W=\mathbb{F}^2\subset V$.  Denote by $x,y,z,w$ the standard basis vectors in $V=\mathbb{F}^4$, so 
$W=\mathrm{Span}_{\mathbb{F}}\{x,y\}$, $S(V)=\mathbb{F}[x,y,z,w]$ and $S(W)=\mathbb{F}[x,y]$. 
We claim that the $G$-invariant 
\[xy(x^2-y^2)zw\] 
in $\mathbb{F}[x,y,z,w]^G$ is indecomposable. 
Consider  the characters $\chi,\psi \in\widehat G$ given by 
\begin{align*}
\chi(a)=\chi(b)=\chi(c)=-1, &  & \psi_{1}(a)=\psi_{1}(b)=1,\psi_{1}(c)=\omega,\\
 &  & \psi_{2}(a)=\psi_{2}(b)=-1,\psi_{2}(c)=\omega
\end{align*} 
and $\mathbb{F}_{\chi}$, 
$\mathbb{F}_{\psi_1}$, $\mathbb{F}_{\psi_2}$ the corresponding one-dimensional $G$-modules. 
Clearly we have the $G$-module isomorphisms 
\[\mathbb{F} z\cong \mathbb{F}_{\psi_1}, \quad \mathbb{F} w\cong \mathbb{F}_{\psi_2}, \quad  \mathbb{F} xy(x^2-y^2)\cong \mathbb{F}_{\chi}.\]  
Since $c$ restricts to an order two transformation of $W$, the modules $\mathbb{F}_{\psi_1}$, $\mathbb{F}_{\psi_2}$ and their duals do not occur as a summand in $S(W)$, it follows that $S(V)^G$ contains no element that has degree $1$ in $z$ and degree $0$ in $w$, and $S(V)^G$ contains no element that has degree $1$ in $w$ and degree $0$ in $z$. 
The equalities $\psi_1\psi_2=\chi$ and $\chi^2=1$ show that 
the elements of $S(V)^G$ having degree $1$ both in $z$ and $w$ are exactly the elements of the form $zwh$ where $h\in S(W)$ and $\mathbb{F} h\cong \mathbb{F}_{\chi}$. 
 The $\overline G$-module structure of $S(W)$ and the structure of 
$S(V)^G\cap S(W)= S(W)^{\overline G}$ is well known from the Shephard-Todd-Chevalley Theorem \cite{chevalley}, \cite{shephard-todd}. We infer that $\mathbb{F}_{\chi}$ occurs with multiplicity one in the degree $4$ component of $S(W)$, namely as the subspace spanned by $xy(x^2-y^2)$, and does not occur in lower degrees. 
This clearly implies that  $xy(x^2-y^2)zw$ is indecomposable in $S(V)^G$. 
\end{example}

\section{Observations and open questions}\label{sec:observations}

\subsection{The Noether number and  multiplicity free representations}\label{subsec:multfree}

When $\mathrm{char}(\mathbb{F})=0$, it was shown in \cite{schmid} that as a consequence of Weyl's Theorem on polarizations, $\beta(G)$ is attained on the regular representation of $G$, which contains each irreducible $G$-module with multiplicity equal to its dimension  
(see \cite{knop} for a variant of Weyl's Theorem \cite{weyl} valid in positive non-modular characteristic). 
However for the small groups studied here, the Noether number is usually attained on  some $G$-modules of much smaller dimensions. Recall that a $G$-module $V$ is {\it multiplicity free} if it is the direct sum of pairwise non-isomorphic irreducible $G$-modules.  

To deal with some particular cases below,  
we need first to state explicitly the following corollary of the proof of the inequality \eqref{also1} given in \cite{CzD:2}: 

\begin{lemma}\label{lemma:induced} 
Let $N$ be a normal subgroup of $G$ with $G/N$ abelian. 
For any $N$-module $W$ there exists a multiplicity free $G/N$-module $U$ such that 
\[\beta(S(U\oplus \mathrm{Ind}^G_NW)^G)\ge \beta(S(W)^N)+\mathsf D(G/N)-1.\] 
In particular, if $\beta(S(W)^N)=\beta(N)$, $\beta(G)=\beta(N)+\mathsf D(G/N)-1$, 
$\mathrm{Ind}_N^G(W)$ is multiplicity free and has no summand on which $N$ acts trivially, 
then $\beta(G)$ is attained on a multiplicity free $G$-module. 
\end{lemma}

\begin{theorem}\label{thm:multfree} 
There exists a multiplicity free $G$-module $V$ such that 
$\beta(G)=\beta(S(V)^G)$ in each of the following cases:  
\begin{enumerate}
\item $G$ is abelian; 
\item $G$ has a cyclic subgroup of index two; 
\item  $G$ has order less than $32$. 
\end{enumerate}
\end{theorem}  

\begin{proof}  
It is sufficient to prove our claim in the case when $\mathbb{F}$ is algebraically closed, so let us assume this. 

1. The case when $G$ is abelian is known. 
Denote by $\widehat G$ the group of characters ($G\to \mathbb{F}^\times$ homomorphisms) of $G$, and for $\chi\in \widehat G$ let $\mathbb{F}_{\chi}$ be the $1$-dimensional $G$-module on which $G$ acts via $\chi$. It is well known 
(see for example  Proposition 4.7 in \cite{CzDG}) that if there exists an irreducible zero-sum sequence over $\widehat G$ of length $\mathsf D(\widehat G)$ with components in $\{\chi_1,\dots,\chi_k\}\subseteq \widehat G$, then
$\beta(G)=\beta(S(\bigoplus_{i=1}^k\mathbb{F}_{\chi_i})^G)$. 

2. For the dicyclic groups $G=Dic_{4m}$ (where $m>1$) we gave an example in the proof of Proposition 9.1 in \cite{CzD:2}  of a $2$-dimensional irreducible $G$-module $V$ 
with $\beta(S(V)^{{Dic}_{4m}})=2m+2$. 

Let $G$ be a non-abelian group with a cyclic subgroup $N$ of index two, and let $\chi$ be a generator of $\widehat N$. Then $\mathrm{Ind}^G_N\mathbb{F}_{\chi}$ is a $2$-dimensional irreducible $G$-module, and so the module $U\oplus \mathrm{Ind}^G_N\mathbb{F}_{\chi}$ 
from Lemma~\ref{lemma:induced} is multiplicity free (being the direct sum of a $2$-dimensional irreducible and the non-trivial $1$-dimensional module). 
Moreover, we have $\beta(S(U\oplus \mathrm{Ind}^G_N\mathbb{F}_{\chi})^G)\ge\beta(N)+1=\frac 12|G|+1$, and $\frac 12|G|+1=\beta(G)$ unless $G$ is dicyclic. 

3. Similar argument works for $C_7\rtimes C_3$, $C_5\rtimes C_4$, and $M_{27}$: 
let $N$ be a maximal cyclic normal subgroup of the  given group, and $\chi$ a generator of $\widehat N$. Then $\mathrm{Ind}_N^G\mathbb{F}_{\chi}$ is an irreducible $G$-module of dimension $|G/N|$, and we are done by Lemma~\ref{lemma:induced}, taking into account the known value of $\beta(G)$ from Section~\ref{sec:table}. 

Now suppose that $G=A\rtimes_{-1} C_2$ is a generalized dihedral group with $A$ being a non-trivial abelian group. Take a minimal subset $\Lambda=\{\chi_1,\dots,\chi_r\}\subset \widehat A$ such that $\mathcal{B}(\widehat A)$ contains an atom of length $\mathsf D(\widehat A)$, all of whose components belong to $\Lambda$. 
By minimality of $\Lambda$, it does not contain the trivial character, and  if $\chi$ and $\chi^{-1}$ both belong to $\Lambda$, then $\chi=\chi^{-1}$. 
If $\chi\neq \chi^{-1}$, then $\mathrm{Ind}^G_A\mathbb{F}_{\chi}$ is irreducible and as an $A$-module is isomorphic to $\mathbb{F}_{\chi}\oplus\mathbb{F}_{\chi^{-1}}$, whereas if 
$\chi=\chi^{-1}$ is non-trivial, then $\mathrm{Ind}^G_A\mathbb{F}_{\chi}$ is the direct sum of two non-isomorphic $1$-dimensional $G$-modules, which as $A$-modules are  isomorphic to 
$\mathbb{F}_{\chi}$.  It follows that for 
$W=\mathbb{F}_{\chi_1}\oplus\cdots \oplus \mathbb{F}_{\chi_r}$ the $G$-module 
$\mathrm{Ind}^G_AW$ is multiplicity free and contains no summands on which $A$ acts trivially. Thus we are done by Lemma~\ref{lemma:induced}. 

The groups $Q_8\times C_2$, $S_3\times C_3$, $C_4\rtimes C_4$,  can also be settled by 
Lemma~\ref{lemma:induced}, taking into account the known value of $\beta(G)$ from Section~\ref{sec:table}.  Indeed, 
for the $2$-dimensional irreducible $Q_8$-module $W$ we have $\beta(S(W))^{Q_8}=6$ 
(see \cite{schmid}) and $\mathrm{Ind}_{Q_8}^{Q_8\times C_2}$ is the direct sum of the two 
non-isomorphic irreducible two-dimensional $Q_8\times C_2$-modules. 
Essentially the same argument can be used for $Dic_{12}\times C_2$: after inducing up to $Dic_{12}\times C_2$ the 
irreducible two-dimensional $Dic_{12}$-module on which the Noether number is attained we get a direct sum of two non-isomorphic irreducible $Dic_{12}\times C_2$-modules. 
Note that $S_3\times C_3\cong C_3\rtimes_{-1} C_6$ and for a non-trivial $\chi\in \widehat C_3$ we have that $\mathrm{Ind}_{C_3}^{C_3\rtimes C_6}\mathbb{F}_{\chi}$ is the direct sum of the three  
pairwise non-isomorphic irreducible $2$-dimensional $S_3\times C_3$-modules. For a generator 
$\chi\in\widehat C_4$, we have that $\mathrm{Ind}_{C_4}^{C_4\rtimes C_4} \mathbb{F}_{\chi}$  
is the direct sum of two non-isomorphic irreducible two-dimensional modules. 

The cases  of the groups $C_3 \rtimes Dih_8$, $A_4 \times C_2$ and $(C_2\times C_2)\rtimes_{-1}C_4$ were settled in Example~\ref{ex:c3dih8}, Example~\ref{ex:a4xc2} and Example~\ref{ex:k4_rtimes_c2}. 

An irreducible module on which the Noether number is attained is given already in the literature for  $A_4$, $\widetilde A_4$ in \cite{CzD:1}, for $H_{27}$ in \cite{Cz2}.  
It was pointed out in  \cite[Example 5.3 and 5.4]{CzDG} that 
the Noether number for $S_4$ is attained on the product of the standard four-dimensional permutation representation  and the sign representation,  and for 
the Pauli group $(C_4\times C_2)\rtimes_{\phi} C_2$ on the direct sum of the two-dimensional pseudo-reflection representation and a one-dimensional representation 
(see also \cite{CzD:3} for some details referred to in \cite{CzDG}).   
\end{proof}

\begin{problem}\label{problem:2} 
Does there exist a group $G$ for which $\beta(S(V)^G)<\beta(G)$ for all multiplicity free $G$-modules $V$? 
\end{problem} 

\begin{remark} 

(i)  By a theorem of Draisma, Kemper and Wehlau \cite{draisma-kemper-wehlau} the universal degree bound for separating invariants is known to be attained on multiplicity free representations. 

(ii) We mention a conjecture of Hunziker \cite[Conjecture~5.1]{hunziker} made for 
reflection groups that has a similar flavor as the  topic of Section~\ref{subsec:multfree}. 
\end{remark}

\subsection{The strict monotonicity of the Noether number}

Since all $G/N$-modules can be viewed as $G$-modules, the inequality 
$\beta(G/N)\le \beta(G)$ holds for any normal subgroup $N$ of any finite group $G$. It was proven by B. Schmid \cite{schmid} that $\beta(H)\le \beta(G)$ for any subgroup $H$ of $G$.  
We shall refer as {\bf S} and {\bf F} for the following conditions on a finite group $G$: 
\begin{itemize}
\item[{\bf S}:] $\beta(H)<\beta(G)$ for each proper subgroup $H$ of $G$. 
\item[{\bf F}:]  $\beta(G/N)<\beta(G)$ for each non-trivial normal subgroup $N$ of $G$.  
\end{itemize} 
It is shown in a subsequent paper \cite{CzD:4} that  conditions {\bf S} and  {\bf F} 
(by generalizing \eqref{also1} for non-abelian $N$) 
hold for all finite groups.  
We collect  in  Theorem~\ref{thm:monotone} below 
facts on  the properties {\bf S} and  {\bf F} that can be read off from the results obtained or quoted  in the present paper.  

\begin{theorem} \label{thm:monotone} 
Condition {\bf S} holds for any finite nilpotent group $G$. 
Moreover, both {\bf S} and {\bf F} hold when 
\begin{enumerate} 
\item $G$ is abelian; 
\item $G$ has a cyclic subgroup of index two; 
\item $G\cong C_p\rtimes C_q$ for odd primes $p,q$ where $q\mid p-1$;  
\item $G$ has order less than $32$. 
\end{enumerate}
\end{theorem}

\begin{proof}  
Suppose first that $G$ is nilpotent, and let $H$ be a proper subgroup of $G$. It is well known that  $H$ is contained as a prime index normal subgroup in a subgroup $K$ of $G$, whence we have 
$\beta(G)\ge \beta(K)\ge \beta(H)+[K:H]-1>\beta(H)$  by \eqref{also1}. 
So {\bf S} holds for $G$. 
Note that if $G$ is abelian, then any factor group of $G$ is isomorphic to a subgroup of $G$, 
hence {\bf F} follows from {\bf S}. 
If $G$ has a cyclic subgroup of index two, then any subgroup or factor group of $G$ has a subgroup of index at most two as well, whence \eqref{indextwo} shows that both {\bf S} and {\bf F} hold. Any non-trivial subgroup or factor group of $C_p\rtimes C_q$ has order $p$ or $q$, and $\beta(C_p\rtimes C_q)\ge p+q-1$ by \eqref{also1}. 

Assume finally that $G$ is a non-abelian group of order less than $32$ that contains  no cyclic subgroup of index two. 
Theorem 1.1 in \cite{CzD:1} asserts that $\beta(H)<\frac 12 |H|$ unless $H$ has a cyclic subgroup of index at most two, or $H$ is isomorphic to one of $C_2\times C_2\times C_2$, $C_3\times C_3$, $A_4$, or $\tilde{A}_4$. Taking into account \eqref{indextwo} 
and the values of the Noether numbers of $C_2\times C_2\times C_2$, $C_3\times C_3$, $A_4$ and $\tilde{A}_4$, this implies that 
$\beta(H)\le 2+\frac 12|H|$ for any non-cyclic $H$, with equality only if $H\cong Dic_{4m}$ is a dicyclic group. Since our $G$  has no cyclic subgroups or factor groups  of order at least $\frac 12 |G|$, we conclude that for any proper subgroup or factor group $H$ of $G$  the inequality $\beta(H)\le \max\{2+\frac 14|G|,\frac 13 |G|\}$ holds, with strict inequality unless $H\cong Dic_{4m}$ is a dicyclic group of order $\frac{1}{2}|G|$ or $H$ is a cyclic group of order $\frac{1}{3}|G|$.  
This immediately implies that {\bf S} and {\bf F} hold for $G$ provided that 
$\beta(G)>\max\{2+\frac 14|G|,\frac 13 |G|\}$. 
From now on assume in addition that $\beta(G)\le \max\{2+\frac 14|G|,\frac 13 |G|\}$. 
So $G$ is one of the following groups from the table in Section~\ref{sec:table}: 
the two groups of order $16$ with Noether number $6$, the group of order $18$ with Noether number $6$, the two groups of order $24$ with Noether number $8$, or the group of order $27$ with Noether number $9$.  
Now $Dih_8\times C_2$ has exactly four elements of order $4$ and no element of order $8$, consequently  does not have $Q_8=Dic_8$ as a subgroup or a factor group, hence 
{\bf S} and {\bf F} hold for this group.  
The group  $(C_2\times C_2)\rtimes C_4$ has  $C_2\times C_2\times C_2$ as a subgroup.  Therefore any order $8$ subgroup  
or factor group of  this group contains $C_2\times C_2$ as a subgroup, and hence it is not dicyclic.  
The group $(C_3\times C_3)\rtimes_{-1}C_2$ has no element of order $6$ and has no dicyclic subgroups or factor groups (as its order is not divisible by $4$). The groups $A_4\times C_2$ and $Dih_{12}\times C_2$ have no element of order $8$ and do not have a subgroup or factor group isomorphic to $Dic_{12}$ (as these 
groups do not have an element of order $4$). Finally, the Heisenberg group does not have an element of order $9$. So {\bf S} and {\bf F} hold for all groups of order less than $32$. 
\end{proof}

\subsection{Dependence on the characteristic}  \label{subsec:char}

It is proved in \cite[Corollary 4.2]{knop} that $\beta^{\mathbb{F}}(G)$ may depend only on the 
characteristic of $\mathbb{F}$, but not on $\mathbb{F}$. 
Therefore we introduce the notation  
\[\beta^{\mathrm{char}(\mathbb{F})}(G)=\beta^{\mathbb{F}}(G).\] 
Moreover, by \cite[Theorem 4.7]{knop} we have $\beta^p(G)\ge \beta^0(G)$ for all primes $p$, and $\beta^p(G)= \beta^0(G)$ holds for all but finitely many primes $p$.  
Knop remarks in \cite{knop} that ``Presently, no group $G$ and prime $p$ not dividing $|G|$ with $\beta^p(G)>\beta^0(G)$ seems to be known''.  This observation  inspires the following question: 

\begin{problem}\label{problem:3}  
Does the equality  $\beta^p(G)=\beta^0(G)$ 
hold for all finite groups $G$ 
and  primes $p$ not dividing $|G|$? 
\end{problem}   

The paper \cite{wehlau} reports as a folklore conjecture that for any permutation $\mathbb{Z}G$-module $V$ (i.e. when $V$ is a free $\mathbb{Z}$-module  with a basis preserved by the action of $G$)  we have 
$\beta(S(\mathbb{F}\otimes_{\mathbb{Z}} V)^G)=\beta(S(\mathbb{Q}\otimes_{\mathbb{Z}}V)^G)$ provided that $\mathrm{char}(\mathbb{F})$ does not divide  $|G|$.  
We note that if this conjecture  is true, then Problem~\ref{problem:3} has a positive answer.  
Indeed, it follows from 
\cite[Theorem 6.1]{knop} and \eqref{eq:noether} that 
if $p$ does not divide $|G|$ then  $\beta^p(G)=\beta(S(\mathbb{F}\otimes _{\mathbb{Z}}V)^G)$ where $\mathbb{F}$ is a field  of characteristic  $p$ and 
$V$ is the direct sum of $|G|$ copies of the regular $G$-module defined over $\mathbb{Z}$ (hence in particular $V$ is a permutation $\mathbb{Z}G$-module).  

In Theorem~\ref{thm:char-indep} below we collect the cases for which Problem~\ref{problem:3} has a positive answer by the results obtained or quoted in the present paper. 

\begin{theorem} \label{thm:char-indep} 
The equality $\beta^p(G)=\beta(G)$ holds for all primes $p$ not dividing the order of $G$ in each of the following cases: 
\begin{enumerate}
\item $G$ is abelian; 
\item $G$ has a cyclic subgroup of index two; 
\item $G$ has order less than $32$.   
\end{enumerate}
\end{theorem} 

\begin{proof} It has been long known that if $G$ is abelian then $\beta^p(G)=\mathsf D(G)$ for all $p\nmid  |G|$. Formula \eqref{indextwo} for the Noether number of a group  with a cyclic subgroup of index two is valid in all non-modular characteristic. 
Finally, the quantities in the table in Section~\ref{sec:table} are independent of the characteristic of the base field (provided that it is non-modular), whence the statement holds also for the remaining groups of order less than $32$. 
\end{proof}

\subsection{ Further observations }

Remark that the groups $S_3\times C_3$ and $(C_3 \times C_3) \rtimes _{-1} C_2$ both have the structure $(C_3 \times C_3) \rtimes _{\alpha} C_2$, the only difference being in the automorphism $\alpha$. However, their Noether numbers are  different. 

The groups $Dih_8\times C_2$, 
$(C_2\times C_2)\rtimes C_4$ and the Pauli group all have the structural description $(C_4 \times C_2) \rtimes_{\alpha} C_2$, only  the automorphism $\alpha$ being different in the three cases. The Noether numbers of the first two are  
equal, and differ from the Noether number of the third one. 

This shows that it would be interesting to understand how $\beta(A \rtimes_{\alpha} B)$ depends on  $\alpha$.

\section{Davenport constants}\label{sec:davenport}

\subsection{The monoid of product-one sequences}\label{sec:monoid}

In this section we introduce further notation related to the small and large Davenport constants of a not necessarily abelian finite group. We follow the presentation of \cite{CzDG}. 

Let  $G_0 \subseteq G$ be a non-empty subset of a finite group $G$. A {\it sequence} over $G_0$
means a finite sequence of terms from $G_0$ which is unordered, and
repetition of terms is allowed (in other words, a sequence over $G_0$  is a multiset of elements from $G_0$). A sequence will be considered as an element of the free abelian monoid $\mathcal F(G_0)$ whose generators are identified with the elements of $G_0$.
We use the symbol ``$\bdot$'' for the multiplication in the monoid $\mathcal F (G_0)$ -- this agrees with the convention in the monographs \cite{Ge-HK06a, Gr13a} -- and we denote multiplication in $G$ by juxtaposition of elements. 
For example, considering elements $g_1,g_2\in G_0$ we have that $g_1 \bdot g_2 \in \mathcal F (G_0)$ is a sequence of length $2$, while $g_1 g_2$ is an element of $G$. Furthermore, 
we use brackets for the exponentiation in $\mathcal F (G_0)$. So for $g \in G_0$, $S \in \mathcal F (G_0)$, and $k \in \mathbb N_0$, we have
\[
g^{[k]}=\underset{k}{\underbrace{g\bdot\ldots\bdot g}}\in \mathcal F (G)  \quad  \text{and}  \quad S^{[k]}=\underset{k}{\underbrace{S\bdot\ldots\bdot S}}\in \mathcal F (G) \,.
\]
Let 
\[
S = g_1 \bdot \ldots \bdot g_{|S|} = \prod_{g \in G_0} g^{[\mathsf v_g (S)]} 
\] 
be a sequence over $G_0$;  here $\mathsf v_g (S)$ is the {\it multiplicity} of $g$ in $S$
and we call $|S|=\sum_{g\in G}\mathsf v_g (S)$ the {\it length} of $S$.  
If $\mathsf v_g(S)>0$, i.e. $S=g\bdot R$ for a sequence $R$ 
with $|R|=|S|-1$, then we 
write  
\[S \bdot g^{[-1]}=R\] 
for the  sequence obtained from $S$ by removing one occurrence of $g$.  
More generally, we write 
$R=S\bdot T^{[-1]}$ if we have $S=R\bdot T$ for some sequences $R,S,T$.  
The identity element  $1_{\mathcal F ( G_0)}$ in 
$\mathcal F (G_0)$ is called the {\it trivial sequence}, and has length  
$|1_{\mathcal F ( G_0)}|=0$.  
We have the usual divisibility relation in the free abelian monoid $\mathcal F (G_0)$ and write $T \mid S$ if $T$ divides $S$.  
A divisor $T$ of $S$ will also be called  a {\it subsequence} of $S$. 
We call $\supp (S) = \{g\in G_0 \mid    \mathsf v_g(S)>0\} \subseteq G_0$ the {\it support} of $S$.
The {\it set of products} of $S$ is 
\[
\pi (S) = \{ g_{\tau (1)} \ldots  g_{\tau (|S|)} \in G \mid \tau\in {\mathrm{Sym}}\{1,\dots,|S|\} \} \subseteq G\]
(if $|S|= 0$, we use the convention that $\pi (S) = \{1_G \}$). 
Clearly, $\pi (S)$ is contained in a $G'$-coset, where 
$G' = [G,G]=\langle g^{-1}h^{-1}g h \mid g, h \in G\rangle$ denotes the commutator subgroup of $G$. 
Set
\begin{align}\label{bigpi}
\Pi (S) = \underset{1_{\mathcal{F}(G_0)} \ne T}{\bigcup_{T \mid S}} {\pi}(T)  \subseteq G .
\end{align}
The sequence $S$ is called a 
\begin{itemize}
\item  {\it product-one sequence} if $1_G \in \pi (S)$,

\item {\it product-one free sequence} if $1_G \notin \Pi (S)$.
\end{itemize}
The set 
\[
\mathcal B (G_0) = \{ S \in \mathcal F (G_0) \colon 1_G \in  \pi (S) \}
\]
of all product-one sequences over $G_0$ is obviously a submonoid of $\mathcal{F}(G_0)$.  
We denote by $\mathcal{A}(G_0)$ the set of {\it atoms} in the monoid $\mathcal{B}(G_0)$. The length of an atom is clearly bounded by 
$|G|$. The {\it large Davenport constant} of $G_0$ is 
\[
\mathsf D (G_0) = \max \{ |S| \colon S \in \mathcal A (G_0) \} \in \mathbb N 
.\]
Moreover, we denote by $\mathcal{M}(G_0)$ the set of product-one free sequences over $G_0$ and we define  the {\it small Davenport constant} of $G_0$ as 
\[
\mathsf d (G_0) = \max \{ |S|  \colon S \in \mathcal M(G_0) \} .  \]
We have the inequality 
\[\mathsf d(G_0)+1\le \mathsf D(G_0)\] 
with equality when the elements in $G_0$ commute with each other. 

\subsection{Some known results}\label{subsec:known-davenport} 

For a non-cyclic group $G$ with a cyclic subgroup of index two  
Olson and White \cite{olson-white} proved that $\mathsf d(G) = \frac 1 2 |G|$.  
Morover,  recently it was  proven by Geroldinger and Grynkiewicz  \cite{GeGryn} that for these groups $\mathsf D(G) = \mathsf d(G) + |G' |$.

For the non-abelian semidirect product $C_p\rtimes C_q$ 
where  $p,q$ are odd primes it was shown by Grynkiewicz \cite[Corollary 5.7 and Theorem 5.1]{Gryn} that we have
$\mathsf d(C_p\rtimes C_q)=p+q-2$ and $\mathsf D(C_p\rtimes C_q)=2p$.

\subsection{The large Davenport constant for $H_{27}$} \label{sec:heisenberg} 
Consider $H_{27}$, the Heisenberg group with 27 elements having the presentation 
$
\langle a,b,c \mid a^3=b^3=c^3=1,\ c= [a,b]=a^{-1}b^{-1}ab\rangle$.
This is an extraspecial group,  its commutator subgroup $\langle c \rangle$ coincides with the center $Z:=Z(H_{27})$. As a result, the commutator identities (which hold for any group) take the following simpler form in this particular case:
\begin{align}\label{comm}
[x,yz] = [x,y][x,z]  \quad [xy,z] = [x,z][y,z]
\end{align}
for any $x,y,z \in H_{27}$. As $[c,x]=1$ for any $x \in H_{27}$ we see that the value of $[x,y]$ depends  only on the cosets $xZ$ and $yZ$ so that the commutator defines in fact a bilinear map on $H_{27}/Z \cong C_3 \times C_3 $ with values in $ C_3$. Moreover as $c$ commutes with every other element of the group it is immediate that every $Z$-coset has a  representative of the form $a^ib^j$ for some $i,j \in \mathbb{Z}/3\mathbb{Z}$ and by repeated applications of \eqref{comm} we get 
\begin{align}\label{det_form}
[a^ib^j, a^kb^l] = c^{il-jk} = c^{\det \left (\begin{smallmatrix} i&j \\ k & l \end{smallmatrix} \right)}.
\end{align}
By \eqref{det_form} the elements $x=a^ib^j$ and $y = a^kb^l$ commute if and only if the vectors $(i,j)$ and $(k,l)$ are linearly dependent over $\mathbb{Z}/3\mathbb{Z}$. 
For the rest we denote by $\bar x$ the image of any $x \in H_{27}$ at the 
natural surjection $H_{27} \to C_3 \times C_3$ and we extend this notation to sequences in the obvious way, as well. 

We say that two sequences $S$ and $T$ over a group $G$ are {\it similar} if  $\alpha(S) = T$ for an automorphism $\alpha \in \Aut(G)$ (the action of $\Aut(G)$ on $G$ extends naturally to an action on $\mathcal{F}(G)$).  
A sequence $S$ over $G$ is called \emph{degenerate} if $\supp(S)$ is contained in a proper subgroup of $G$. 

\begin{lemma} \label{grynk}
Let $T=R\bdot S$ be a product-one sequence such that $|\pi(R)|=3$  and $S$ is not product-one free. Then $T$ is not an atom. 
\end{lemma}

\begin{proof} 
By the assumption on $S$ there is a non-empty product-one sequence $U\mid S$. Consider the sequence  $V=T\bdot U^{[-1]}$. Then $\bar{V}$ is a zero-sum sequence 
over $C_3\times C_3$, whence $\pi(V) \subseteq Z$. But for $R \mid V$ we have $|\pi(R)|= 3$ hence  $\pi(V) = Z$ so that $V$ is also a product-one sequence. 
Thus the equality $T=U\bdot V$ shows that $T$ is not an atom.  
\end{proof}

\begin{lemma} \label{fullcoset}
Let $T$ be a non-degenerate sequence  over $H_{27} \setminus Z$ of length at least $3$.   
Then either $|\pi(T)|=3$ 
or  $\bar T= \bar e \bdot \bar f \bdot (-\bar e{-\bar f})$ for a basis $\{\bar e,\bar f\}$ of $H_{27}/Z$
\end{lemma}
\begin{proof}
Assume that $\pi(T)$ is not a full $Z$-coset. As $T$ is non-degenerate 
there must be two elements $e,f$ in $T$ such that $[e,f]=c$.  
Hence $|\pi(T)|= 2$. 
Let $g$ be an arbitrary element in $T\bdot (e\bdot f)^{[-1]}$.   
Then  $\bar g \neq \bar e$ because
otherwise $\pi (g \bdot e \bdot f ) \supseteq \{ gef, gfe, fge\} =gefZ$. 
Also $\bar g \neq -\bar e$ because
otherwise $\pi (g \bdot e \bdot f ) \supseteq \{ gef, gfe, efg\} =gefZ$. 
Similarly $\bar g$ is different from $\bar f$, $-\bar f$. 

As a result in the $\mathbb Z/ 3\mathbb Z$-vector space   $H_{27}/Z \cong C_3 \times C_3$ we have a relation 
$\alpha \bar e + \beta \bar f + \gamma \bar g=0$ 
where the coefficients $\alpha, \beta $ are non-zero. 
Moreover $\gamma \neq 0$ also holds by the linear independence of $\bar e$ and $\bar f$.
Up to similarity and the choice of the basis $e,f$ only two cases are possible:
(i) $ \bar e + \bar f = \bar g $, 
but then $[e,f]= [e,g]=c$, hence $\pi(e \bdot f \bdot g) = efgZ$, again a contradiction, or
(ii) $\bar e + \bar f + \bar g =0$; 
then  $[e,f] = [g,e] =[f,g]=c$ hence 
$efg = c  feg = fge = cgfe = gef = c egf$,
so that
$|\pi(e \bdot f \bdot g)| = 2$.
\end{proof}

In the proof below we shall use the following ad hoc terminology: 
A subset of $C_3 \times C_3$ of the form $\{ e,f,-e-f\}$ where $e,f$ form a basis of $C_3\times C_3$ will be called  an {\it affine line}, while a subset  of the form $\{ e,f,e+f\}$ where $e,f$ form a basis of $C_3\times C_3$
will be called  an {\it affine cap} (the terminology is motivated by the literature on the so-called {\it cap set problem}). Note that a three-element subset of $C_3\times C_3$ in which any two elements are linearly independent over 
$\mathbb{Z}/3\mathbb{Z}$ is either an affine line or an affine cap.

\begin{proposition} \label{H27large}
$\mathsf D (H_{27}) \le 8$. 
\end{proposition}

\begin{proof}
Assume indirectly that there is an atomic product-one sequence $T$ of length at least $9$.  After ordering the elements of $T=g_1 \bdot g_2 \bdot\bdot\bdot g_n$ in such a way that $g_1g_2\cdots g_n=1$ and replacing $T$ with $T' = g_1 \bdot \bdot \bdot g_8 \bdot ( g_9g_{10} \cdots g_n)$ we may assume that $|T| =9$.

{\bf A.} If $T$ is degenerate then it is in fact an irreducible zero-sum sequence over $C_3 \times C_3$, hence $|T| \le \mathsf D(C_3 \times C_3) =5$, a contradiction. 
So for the  rest we  assume that $T$ is non-degenerate so that $\bar T$ contains a basis $\{e,f\}$ of $H_{27}/Z$.  

{\bf B.}  
$\bar T$ contains an element $g \not\in \langle e \rangle \cup \langle f \rangle$ (i.e. $\bar T$ has an affine line or an affine cap as a subsequence).  
Otherwise, if $\bar T$ is contained in the set $\langle e \rangle \cup \langle f \rangle$ then  $T = C  \bdot A_1 \bdot \dots \bdot A_t$ where $\bar C = 0^{[k]}$ and $\bar A_i \in \{e^{[3]},  f^{[3]},-e\bdot e , -f \bdot f \}$. 
Choose $a_i\in\pi(A_i)$ for  $i=1,\ldots,t$. Then  the sequence  $Q:=C \bdot a_1 \bdot \ldots \bdot a_t$ is a zero-sum sequence over $Z \cong C_3$, hence if $k+t > 3$ then $Q$ factors into the product of two non-empty zero-sum sequences and $T$ factors accordingly,  a contradiction.  If $k >0$ then  we get $k+t \ge k+ (9-k)/3 >3$, as $|A_i| \le 3$ for  all $i$,
again a contradiction.  Hence $k=0, t=3$ and $\bar T$ is similar to  $e^{[6]}\bdot f^{[3]}$. Then  $T$ contains a degenerate subsequence of length $6 > \mathsf D(C_3 \times C_3)$, which in turn  must contain a proper zero-sum subsequence $R$ such that $\bar R = e^{[3]}$. Then the complement $S = T \bdot R^{[-1]}$ has $|S|=6$, hence $\pi(S) = Z$ by Lemma~\ref{fullcoset}, a  contradiction by Lemma~\ref{grynk}.

{\bf C.} 
 $\bar T$ cannot contain a subsequence $\bar R \mid  \bar T$ 
similar to one of the following sequences: 
 \begin{align}
& e\bdot  f \bdot (-e) \bdot (-f) \label{case1}\\
& e \bdot f \bdot (e+f)^{[2]} \label{case2}
\end{align}
Indeed, these sequences $\bar R$ are such that for  their preimages $R$  we have that $\pi(R) = Z$ by Lemma~\ref{fullcoset}. So for any such $R \mid  T$   the complement $S = T \bdot R^{[-1]}$ must be product-one free, as otherwise we would get a contradiction by Lemma~\ref{grynk}.  Hence  $S$ cannot be degenerate, as $|S|=5 =\mathsf D(C_3 \times C_3)$ and 
$\pi(S) \subseteq Z \setminus \{1 \}$. Therefore by Lemma~\ref{fullcoset} it is necessary that $S = x \bdot y \bdot L$ where $\bar x = \bar y =0$ and $\bar L$ is an affine line. 
Moreover by our assumption we must have $\pi(L) =Z \setminus \{ 1\}$ and $x,y \in Z \setminus \{1\}$, as well. But then $1\in \pi(x \bdot L)$, a contradiction.

 {\bf D.}
 We claim that $\bar T$ must contain an affine line and $0\not\in\mathrm{supp}(\bar T)$. For assume that $\bar T$ does not contain an affine line.  
 Then $\bar T$ still contains an affine cap by {\bf B}, so we may assume that $e\bdot f \bdot (e+f)\mid T$.
 Then to avoid affine lines it is necessary that $e-f,f-e,-e-f$ do not occur in $\bar T$. 
To avoid subsequences of type \eqref{case1}  we have either $\mathsf v_{\bar T}(-e)=0$ or 
$\mathsf v_{\bar T}(-f)=0$. 
Assume now that $\mathsf v_{\bar T}(-e)>0$ and $\mathsf v_{\bar T}(-f)=0$ 
(the case    $\mathsf v_{\bar T}(-e)=0$ and $\mathsf v_{\bar T}(-f)>0$ is analogous). 
Then to avoid subsequences of type \eqref{case2} we must have $\mathsf v_{\bar T}(f) = \mathsf v_{\bar T}(e+f) = 1$. As a result the image of $\bar T$ modulo $\langle e \rangle$ is $f^{[2]}\bdot 0^{[7]}$, but this is not a product-one sequence  in $C_3\times C_3 / \langle e \rangle \cong C_3$, a contradiction.

So it remains that $\bar T = 0^{[k]} \bdot e^{[i]} \bdot f^{[j]} \bdot (e+f) $.
Since $\bar T$ is a zero-sum sequence over $C_3 \times C_3$ we must have $i \equiv j \equiv 2 \pmod 3$. 
As $|\bar T| = 9$  either $k + i =6$ or $k+j =6$ and we get again a degenerate subsequence 
which must contain a zero-sum subsequence $R$ such that its complement $S = T \bdot R^{[-1]}$ has $\pi(S) = Z$ by Lemma~\ref{fullcoset}, again a contradiction. 
Thus $\bar T$ contains an affine line. An argument similar to the one in {\bf C} shows also that $0$ does not occur in   $\bar T$. 

From now on we  assume that $\bar T$ contains the affine line $\bar L =  \{e,f,g\}$ as a subsequence. 

{\bf E.} 
Next we claim that $\sum_{x \in -\bar L}\mathsf v_{\bar T}(x) \le 1$. 
Indeed, $\bar T$ cannot contain two different elements $x,y$ belonging to $- \bar L$ because that would yield a subsequence of type \eqref{case1}. Moreover, 
$\mathsf v_{\bar T}(x) \ge 2$ for some $x$ in $-L$ would yield  the type \eqref{case2}  subsequence 
$x\bdot x\bdot L\bdot (-x)^{[-1]}$  of $\bar T$.

{\bf F. }
We claim that $\mathsf v_{\bar T}(e-f)=0$ and $\mathsf v_{\bar T}(f-e)=0$. 
For assume to the contrary that $\mathsf v_{\bar T}(e-f)>0$  (the case $\mathsf v_{\bar T}(f-e)>0$  being analogous). 
Observe that for any $x$ in  $L$ we must have $\mathsf v_{\bar T}(x)  =1$  to avoid subsequences of type \eqref{case2}. 
In view of {\bf E} it follows that $\mathsf v_{\bar T}(e-f)+ \mathsf v_{\bar T}(f-e)\ge |T|-4 =5=\mathsf D(C_3\times C_3)$,
so that $T$ has again a degenerate subsequence containing a zero-sum subsequence $R$ (with $\mathrm{supp}(\bar R)\subseteq\{0,e-f,f-e\}$) such that its complement $S = T \bdot R^{[-1]}$ has $\pi(S) =Z$ by Lemma~\ref{fullcoset}, leading to a contradiction by Lemma~\ref{grynk}.

{\bf G.}
Now we prove that $\supp(\bar T) = L$. For otherwise it would follow from {\bf D, E, F}  that $\bar T = (-x) \bdot x \bdot \bar T_0$ where $\mathsf v_L(x)>0$ and $\supp (\bar T_0) \subseteq L$. But $\bar T_0$ must also be a zero-sum sequence, hence $|\bar T_0|$ is divisible by $3$
because  the only irreducible zero-sum sequences over $C_3\times C_3$ with support contained in $L$ are $e^{[3]}, f^{[3]}, (-e{-f})^{[3]}, e\bdot f\bdot (-e{-f})$.
 But then $|T| \equiv 2 \mod 3$, contradicting the assumption that $|T|=9$.

 {\bf H.}
So we have a factorization $T= L\bdot R$ such that $\bar L = e \bdot f \bdot g$ with $e+f+g=0\in C_3\times C_3$ and $\supp(\bar R) \subseteq \{e,f,g \}$.
For any such factorization $\bar L$ and $\bar R$ are  zero-sum sequences, hence $\pi(R) \subseteq Z$, moreover $ |\pi(R)|=3$ by Lemma~\ref{fullcoset}, so that $1 \in \pi(R)$ and consequently $1 \not\in \pi(L)$. 
Observe now that an element in $ \supp(\bar T)$, say $e$ must have $\mathsf v_{\bar T}(e) \ge 3$. Let $Q = x_1 \bdot x_2 \bdot x_3 \mid T$ be such that $\bar Q = e^{[3]}$ and $x_1 \in L$. 
If $x_i \neq x_1$ for some $i >1$ then consider the sequence $L ' := L \bdot x_1^{[-1]} \bdot x_i$. By assumption $\pi(L) = Z \setminus \{ 1\}$ and $x_1^{-1} x_i \in Z \setminus \{1\}$ hence $\pi(L') = \pi(L)x_1^{-1} x_i \ni 1$. Thus $\bar L'=\bar L$ and $T=L'\bdot R'$would be a factorization of $T$ as a product of two product-one sequences. 
We conclude that $x_1= x_2 =x_3$, so that $Q$ is a product-one subsequence
and $T=Q\bdot S$ where $\bar S$ is a product-one sequence over $C_3\times C_3$ and hence $S$ 
is a product-one sequence by Lemma~\ref{fullcoset}. This contradicts the assumption that $T$ is an atom. 
\end{proof}

\begin{remark} 
The idea of Lemma~\ref{grynk} appears in the proof of a result of Grynkiewicz \cite[Corollary~3.4]{Gryn} which yields for the Heisenberg group of order $p^3$ the inequality 
\begin{align}\label{gryneq}\mathsf D(G) \le \mathsf d(G) + |G'| =\mathsf d(G)+p.\end{align} 
We established by our algorithm that  $\mathsf d(H_{27}) =6$ so in this case \eqref{gryneq}  gives $\mathsf D(G) \le 9$ which in view of  Proposition~\ref{H27large} is not sharp. On the other hand we show below that Proposition~\ref{H27large} is sharp.
\end{remark}

\begin{proposition}\label{prop:lower8}
$\mathsf D(H_{27}) \ge 8$.
\end{proposition}

\begin{proof}
Assume  $c \in \pi(a^{[i]} \bdot b^{[j]})$ for some $i,j \in \mathbb{N}$.  Then  $i,j >0$ since $c \not\in \langle a \rangle$ and $c \not\in \langle b \rangle$. Moreover $\bar c = 0 =\bar a^i \bar b^{j}$ holds, as well. But since $\bar a$ and $\bar b$ are independent in $C_3 \times C_3$ it is necessary that $i\equiv j \equiv 0 \mod 3$. It follows that $i+j \ge 6$.  On the other hand $c= [a,b] = aabbab$ hence the minimal  expression of $c$ in terms of $a$ and $b$ has length $6$. By a similar argument  the same is true for $c^{-1}$, as well. 

Now $ccbbaaba=1$ and we claim that the product-one sequence $S= a^{[3]} \bdot b^{[3]} \bdot c^{[2]}$ is an atom.  Otherwise let $S = T \bdot R$ for some  non-empty product-one sequences $T,R$ and we may assume that $\mathsf v_T(c)>0$. If $\mathsf v_R(c)>0$ holds, too, then $T \bdot c^{[-1]}$ contains only $a, b$ and $c^{-1} \in \pi(T \bdot c^{[-1]})$. Hence, by what has been  said before, $T \bdot c^{[-1]} = a^{[3]} \bdot b^{[3]}$. But then $R=c$, which is not a product-one sequence.  From this contradiction it follows that $\mathsf v_R(c)=0$ hence $c^{[2]} \mid T$. But then $\mathrm{supp}(T\bdot c^{[-2]})\subseteq \{a,b\}$  and  $c \in \pi(T\bdot c^{[-2]}) $, hence again $T\bdot c^{[-2]}= a^{[3]} \bdot b^{[3]}$,  so that $R$ is empty, which is a contradiction.
\end{proof}

\subsection{Diameter of the Cayley digraph}
Next we state a general inequality that can be obtained by 
a similar argument as  the proof of Proposition~\ref{prop:lower8}.  
This result is not used here, but seems worthwhile to mention, as it involves the diameter of Cayley digraphs, which (in contrast with the large Davenport constant for a non-abelian group) has a rather extensive literature:  

\begin{proposition}
Let $X$ be a set of generators of a finite group $G$ and $\mathrm{Cay}(G,X)$ the corresponding Cayley digraph.
Then $\mathsf D(G)  \ge \mathrm{diam}(\mathrm{Cay}(G,X))+1$.
\end{proposition}

\begin{proof}
Let $1=g_0,g_1,\ldots,g_d \in G$ be a sequence of vertices on a non-self-crossing path  of maximal length in  $\mathrm{Cay}(G,X)$, so that $d= \mathrm{diam}(\mathrm{Cay}(G,X))$.  Now set $s_i = g_{i-1}^{-1}g_{i} $ for all $i= 1, \ldots, d$  and consider the sequence $S= s_1 \bdot \ldots \bdot s_d \bdot g_d^{-1}$. By construction $S$ is a product-one sequence and we claim that it is an atom in $\mathcal B(X \cup \{ g_d^{-1}\})$. For otherwise $S= A\bdot B$ for some product-one sequences $A, B$ where $g_d^{-1} \in A$, and then we can order the elements of $A$ in such a way that we obtain an equality $s_{i_1} \cdots s_{i_r} g_d^{-1} =1$ where $r<d$.  But since $s_i \in X$ for all $i$, this means that we have in $\mathrm{Cay}(G,X)$ a path of length $r$ from $1$ to $g_d$, a contradiction. 
\end{proof}

\begin{problem}
Is there any relation between $ 
\max_{X\subseteq G}\{\mathrm{diam}(\mathrm{Cay}(G,X))\}$ and $\mathsf d(G)$? 
More concretely, based on a little computer experimentation we raise the following question: 
does the inequality $\mathsf d(G) \ge \mathrm{diam}(\mathrm{Cay}(G,X))$ hold in general?  
\end{problem}

\section{Computing the Davenport constants}\label{sec:algorithms}
We take over  the following notations from Section~\ref{sec:monoid}: let $\mathcal F =\mathcal F(G)$ be the monoid  of all sequences of elements in $G$, let $\mathcal M \subseteq \mathcal F$ be the set of product-one free sequences and  $\mathcal A \subseteq \mathcal F$ the set of atoms, i.e. all product-one sequences which cannot be written as the product of two non-empty product-one subsequences. For any $k\ge 1$ let $\mathcal F_k \subseteq \mathcal F$ denote the set of all sequences of length $k$ and set $\mathcal M_k = \mathcal M \cap \mathcal F_k$, $\mathcal A_k = \mathcal A \cap \mathcal F_k$.
Moreover, 
$\mathcal M_0=\{1_{\mathcal F}\}$ 
consists of the empty sequence. 

We propose two algorithms: one for enumerating all product-one free sequences and another one for computing all the atoms 
and thereby establishing the values of $\mathsf d(G)$ and $\mathsf D(G)$, respectively.  We have implemented these algorithms in the GAP computer algebra system (see \cite{GAP4}) and used them to compute and/or verify the small and large Davenport constants given in Section~\ref{sec:table}.

The algorithms construct the set $\mathcal{M}$ (respectively $\mathcal{A}$) successively, as a union of the sets $\mathcal{M}_k$ (respectively $\mathcal{A}_k$). During the construction of the sets $\mathcal M_k$  or $\mathcal A_k$ we are not testing all elements of $\mathcal F_k$ for being product-one free or atomic (as this task would be practically unfeasible), but we limit the scope of our search to a much smaller subset defined in what follows. 

A sequence $T$ is called \emph{a splitting} of a sequence $S$  if  $T = S \bdot g^{[-1]} \bdot x \bdot y$ for some $g\in \supp(S)$ and $x,y \in G$ such that $xy=g$. We denote by $S \prec T$ the fact that $T$ is a splitting of $S$. For any subset $\mathcal S \subset \mathcal F(G)$ we set $\gamma(\mathcal S) := \{ T \in \mathcal F(G) \colon S\prec T \text{ for some } S \in \mathcal S  \}$.  

\begin{lemma}\label{splitting}
For any $k \ge 1$ we have $\mathcal M_{k+1} \subseteq \gamma(\mathcal M_k)$ and $\mathcal A_{k+1} \subseteq \gamma(\mathcal A_k)$. 
\end{lemma}
\begin{proof}
If $T=g\bdot h\bdot T'$ is a product-one free (respectively atomic) sequence of length $k+1\ge 2$, then  
obviously the sequence 
$S=(gh)\bdot T'$ has length $k$ and is  product-one free (respectively atomic) too. Moreover, 
$S \prec T$. 
\end{proof}

\begin{remark}\label{happyend}
Observe that by this lemma if $\mathcal M_k = \emptyset$ for some $k \ge 1$ then $\mathcal M_l = \emptyset$  for every $l >k$. So the smallest $k$ such that $\mathcal M_k = \emptyset$ equals  $ \mathsf d(G)+1$. Similarly the smallest $k$ such that $\mathcal A_k = \emptyset$ is equal to $\mathsf D(G)+1$.
\end{remark}

Another idea used for reducing the complexity of our algorithms was to enumerate the elements of $\mathcal M_k$ and $\mathcal A_k$ only ``up to similarity''. 
Recall from Section~\ref{sec:heisenberg} that two sequences $S$ and $T$ are said to be similar if there is an  automorphism $\alpha \in \Aut(G)$ such that $S = \alpha(T)$.

\begin{lemma}\label{reps}
Let $\mathcal R_k$ be a set of $\Aut(G)$-orbit representatives over $\mathcal M_k$ (respectively $\mathcal A_k$). 
Then there is a system of $\Aut(G)$-orbit representatives $\mathcal R_{k+1}$ over $\mathcal M_{k+1}$ (respectively $\mathcal A_{k+1}$) such that $\mathcal R_{k+1} \subseteq \gamma(\mathcal R_k)$. 
\end{lemma}

\begin{proof}
Observe that the relation $\prec$ is compatible with the $\Aut(G)$-action in the sense that $S \prec T$ holds if and only if $\alpha(S) \prec \alpha(T)$ for some $\alpha \in \Aut(G)$ if and only if $\alpha(S) \prec \alpha(T)$ for each $\alpha \in \Aut(G)$. 

Now we prove that for any  sequence $T \in \mathcal M_{k+1}$ its $\Aut(G)$-orbit $\Orb(T)$ has a non-empty intersection with $\gamma(\mathcal R_k)$. 
Take an arbitrary sequence $S\prec T$ so that $S \in \mathcal M_k$. 
As $\mathcal R_k\subseteq  \mathcal M_k$ is a complete set of representatives, there is some $R \in \mathcal R_k$ such that $R =\alpha (S)$ for some $\alpha \in \Aut(G)$.
Hence by the compatibility of $\prec$ it follows that $ R = \alpha(S) \prec \alpha(T) $. As a result $\alpha(T) \in \Orb(T) \cap \gamma(\mathcal R_k)$ and we are done. The proof is similar for the case of $\mathcal A_k$.
\end{proof}

\subsection{Algorithm for computing the small Davenport constant} \label{sec:alg_small_dav} 

\renewlist{itemize}{itemize}{20} \setlist[itemize]{label=$\cdot$}

Let us record the following obvious statement: 

\begin{lemma}\label{oneprodfreetest}
A sequence  $R$ of length $k\ge 1$ 
is product-one free if and only if  $1\notin \pi(R)$ and for all 
$g\in \supp(R)$ we have that $R\bdot g^{[-1]}\in \mathcal{M}_{k-1}$.
\end{lemma}

\medskip{}
\begin{algorithm}

	\medskip{}
	
	\SetKw{LogicAnd}{and} \SetKw{LogicOr}{or}
	\SetKwInOut{Input}{input} \SetKwInOut{Output}{output}
	
	\Input{$G$}
	\Output{$\mathsf d(G)$ }
	
	\BlankLine
	\BlankLine
	$\mathcal{M}_0 \leftarrow \{1_{\mathcal{F}_G}\}$\;
    $\pi(1_{\mathcal{F}_G}) \leftarrow \{1_G\}$\;
	$\mathcal{M}_1 \leftarrow G\setminus\{1_G\}$\label{line3}\;
    $\mathcal{R}_{1}\leftarrow \{g_\alpha : \alpha \in \mathcal{M}_1/\Aut(G)\}$\label{line4}\;
	\lForAll{$g\in \mathcal{M}_1$}{$\pi(g)\leftarrow \{g\}$}\label{line5}
	$k\leftarrow 1$\;
	\Repeat{$\mathcal{M}_k=\emptyset$}{
		$k \leftarrow k+1$\;
    	$\mathcal{M}_k\leftarrow \emptyset$\;
        $\mathcal{R}_k\leftarrow \emptyset$\;
		\ForAll{$S\in \mathcal{R}_{k-1}$}{
			\ForAll{$g\in \supp(S)$}{
				\ForAll{$x\in G\setminus \{1_G, g\}$}{
                	$R \leftarrow S\bdot g^{[-1]}\bdot x \bdot (x^{-1}g)$\label{line16}\;
					\If{$R\notin \mathcal{M}_k$ \LogicAnd $\forall g \in \supp(R): R\bdot g^{[-1]}\in \mathcal{M}_{k-1}$\label{line15}}{
                        
                            $\pi(R)\leftarrow\bigcup_{R_1\bdot R_2 = R\bdot x^{[-1]}}\left(\pi(R_1)\,x\,\pi(R_2)\right)$\label{line17}\;
                            \If{$1_G\notin \pi(R)$\label{line18}}{
                                $\mathcal{R}_k\leftarrow \mathcal{R}_k\cup \{R\}$\label{line19}\;
                                $O_R\leftarrow \{R\}$\;
                                \ForAll{$\alpha \in \Aut(G)$}{
                                	$R'\leftarrow \alpha(R)$\;
                                    \If{$R'\notin O_R$}{
                                    	$O_R\leftarrow O_R\cup \{R'\}$\;
                          				$\pi(R')\leftarrow \alpha\left(\pi(R)\right)$\label{line24}\;
                                    }
                                }
                                $\mathcal{M}_k\leftarrow \mathcal{M}_k\cup O_R$\label{line20}\;
                            }
                        
					}
				}
			}
		}
	}\label{line33}
	
	\Return $k-1$\;
	
	\BlankLine
	
\caption{SmallDavenport($G$) \label{smalldavenportalgorithm}} 
\end{algorithm}
\medskip{}

Given an arbitrary finite group $G$, Algorithm \ref{smalldavenportalgorithm} computes the small Davenport constant. The set of all product-one free sequences of $G$ may be obtained as $\mathcal{M}=\bigcup_k \mathcal{M}_k$ after the repeat-until loop has terminated (right after line \ref{line33}).

In lines \ref{line3}--\ref{line5} of the algorithm we have identified the sequences of length one with the respective group elements. Hence the set $\mathcal{M}_1$ of product-one free sequences having length one consists of all the group elements except the identity (line \ref{line3}), from which we choose a set $\mathcal{R}_1$ of $\Aut(G)$-orbit representatives (line \ref{line4}). For every length-one sequence we store the set of products (line \ref{line5}). These initial data having been established, the algorithm proceeds with computing the sets $\mathcal{M}_k$ and $\mathcal{R}_k$ for all increasing $k>1$, until it finds the first value $k$ for which $\mathcal{M}_k=\emptyset $. The finiteness of the algorithm is guaranteed by Lemma \ref{splitting} (together with Remark \ref{happyend}). 

Observe that inside the for-loops in line \ref{line16} the variable $R$ will take successively all the values from the set $\gamma(\mathcal{R}_{k-1})$, therefore it is guaranteed that for $k>1$ all the possible candidate sequences will be tested for being product-one free and the set $\mathcal{R}_k$ (from Lemma \ref{reps}) is built correctly. 

Before testing the sequence $R$ for being product-one free (using Lemma \ref{oneprodfreetest}), in line \ref{line15} it is checked first whether the sequence $R$ is among the already computed product-one free sequences of length $k$ (via the condition $R\notin \mathcal{M}_k$). Elimination of duplicates is essential for correctly building the set $\mathcal{R}_k$ of representatives (so that in line \ref{line19} we compute only the set $\mathcal{R}_k \subseteq \gamma(\mathcal{R}_{k-1})$ and not one of its supersets). If $R$ is a new sequence (in the sense that it is not similar to any of the already found product-one free sequences), it gets tested using the conditions from Lemma \ref{oneprodfreetest} (the second condition from line \ref{line15} and the condition from line \ref{line18}). $\pi(R)$ is computed (in line \ref{line17}) using the following observation: for any sequence $R$ and fixed $x\in R$, we can write 
\[\pi(R) = \bigcup_{R_1 \bdot R_2 = R \bdot x^{[-1]}} \pi(R_1)\,x\,\pi(R_2).\] 
If $R$ is found to be product-one free, it is also added as a new representative of the 
$\Aut(G)$-orbits (line \ref{line19}). The orbit itself is computed successively (and stored in the variable $O_R$) by evaluating all the automorphisms $\alpha$ on the sequence $R$. In line \ref{line20} the set $O_R$ will be the same as the set $\Orb(R)$ of $\Aut(G)$ orbits of $R$. 

For every new sequence $R'$ (new in the sense of not being in the already computed part of the orbit), we also have to compute and store the set $\pi(R')$ by acting with the proper automorphism $\alpha$ on the set $\pi(R)$ (line \ref{line24}). Hence the algorithm stores the product sets for all the enumerated product-one free sequences. This is essential for a relatively fast computation performed in line \ref{line17}. 

The set $\mathcal{M}_k$ gets computed in line \ref{line20} as a union of $\Aut(G)$ orbits of the representatives. In the last step, the value $k-1$ is returned: this is the small Davenport constant of $G$, as explained in Remark \ref{happyend}.

\subsection{Algorithm for computing the large Davenport constant}\label{subsec:alg_large_dav}

Given an arbitrary finite group $G$, Algorithm \ref{largedavenportalgorithm} computes the large Davenport constant. The set of all atoms of $G$ may be obtained as $\mathcal{A}=\bigcup_k \mathcal{A}_k$ after the repeat-until loop has terminated (right after line \ref{Line19}).

The algorithm begins with loading the sets $\mathcal{A}_1$ and $\mathcal{R}_1$ with the sequence consisting only of the identity element of the group. In line \ref{Line11} the variable $S'$ will take successively all the values from the set $\gamma(\mathcal{R}_{k-1})$, therefore it is guaranteed that for $k>1$ all the possible candidate sequences will be tested for being atoms and the set $\mathcal{A}_k$ (from Lemma \ref{reps}) is built correctly. In line \ref{Line10} the variable $h$ is prevented from taking the values $1_G$ or $g$, since this would bring in the identity element in the support of $S'$.

The algorithm first checks whether a candidate sequence $S'$ is among the already computed atoms of length $k$ via the condition $S' \notin \mathcal{A}_k$ from line \ref{Line12}, ensuring that $S'$ is a new atom (in the sense that it is not similar to any of the already found atoms of length $k$). Checking whether a sequence $S'$ is an atom is performed by the second condition in line \ref{Line12}: by definition the product-one sequence $S'$ is an atom if it cannot be written as a product of shorter atoms. The predicate $\mathbf{D}(S',k)$ in line \ref{Line12} is true if and only if the product-one sequence $S'$ of length $k$ is not an atom and may be given in the following way:  
\begin{align*}\mathbf{D}(S',k)\iff&\exists A \in \bigcup_{i=2}^{\left\lfloor \frac{k}{2} \right\rfloor} \mathcal{A}_i \text{ with }\ A\mid S'\text{ such that} \\  &  S'\bdot A^{[-1]}\in \mathcal{A}_{k-|A|} \textrm{ or } \mathbf{D}(S'\bdot A^{[-1]},k-|A|).\end{align*}
It may be implemented using a recursive function, bailing out with a positive answer as soon as it finds a decomposition of $S'$ into atoms.

In line \ref{Line13} the newly found atom $S'$ gets added to the set $\mathcal{R}_k$ of representatives and the set $\mathcal{A}_k$ of atoms is completed with the $\Aut(G)$ orbit of $S'$, denoted by $\Orb(S')$ (in line \ref{Line14}). In the last step, the value $k-1$ is returned as the large Davenport constant of $G$, as explained in Remark \ref{happyend}.

Although the check is performed only for a system of representatives of the $\Aut(G)$ orbits of atoms (for all $S'\in \mathcal{R}_{k-1}$), the ``decomposability test'' described by the predicate $\mathbf{D}(S',k)$ uses the whole set of already computed atoms $\bigcup_{i=2}^{k-1}\mathcal{A}_i$, which is the main performance bottleneck of the algorithm.

\medskip{}
\begin{algorithm}[H]

	\medskip{}
	
	\SetKw{LogicAnd}{and} \SetKw{LogicOr}{or} \SetKw{LogicNot}{not}
	\SetKwInOut{Input}{input} \SetKwInOut{Output}{output}
	
	\Input{$G$ }
	\Output{$\mathsf D(G)$ }
	
	\BlankLine
	\BlankLine
	
	$\mathcal{A}_{1}\leftarrow \{1_G\}$\;
	$\mathcal{R}_{1}\leftarrow \{1_G\}$\;
	$k\leftarrow 1$\;
	
	\Repeat{$\mathcal{A}_k=\emptyset$}{
    	$k\leftarrow k+1$\;
		$\mathcal{A}_k\leftarrow\emptyset$\;
		$\mathcal{R}_k\leftarrow \emptyset$\;
		\ForAll{$S\in\mathcal{R}_{k-1}$}{
			\ForAll{$g\in \supp(S)$}{
				\ForAll{$x\in G\setminus \{1_G, g\}$\label{Line10}}{
                    $S'\leftarrow S\bdot g^{[-1]} \bdot x \bdot (x^{-1}g)$\label{Line11}\;
					\If{$S' \notin \mathcal{A}_k$ \LogicAnd \LogicNot $\mathbf{D}(S',k)$\label{Line12}}{
						$\mathcal{R}_{k}\leftarrow \mathcal{R}_{k}\cup\{S'\}$\label{Line13}\;
						$\mathcal{A}_{k}\leftarrow \mathcal{A}_{k}\cup \Orb(S')$\label{Line14}\;
					}
				}
			}
		}
	}\label{Line19}
	
	\Return $k-1$\;
	
	\BlankLine
	
\caption{LargeDavenport($G$)} \label{largedavenportalgorithm}
\end{algorithm}
\medskip{}

\begin{remark}
We have used a GAP implementation of Algorithm \ref{smalldavenportalgorithm} and \ref{largedavenportalgorithm} to compute the small and large Davenport constants from the table in Section~\ref{sec:table} in case of those groups for which a formula for these constants in not readily available (see Subsection~\ref{subsec:known-davenport}).

Although parallelization is easily achievable in case of both algorithms, a ``single-threaded'' implementation was found to fit our present requirements, using a personal computer clocked at 2 GHz.

The small Davenport constants from Section~\ref{sec:table} have been obtained in under a minute worth of computation, with the single exception of the group $M_{27} = C_9 \rtimes C_3$, for which $\mathsf{d}(M_{27})=10$ has been obtained in about 7 minutes. This group was found to have a number of $102212$ product-one free sequences, grouped in $1987$ equivalence classes. In contrast, the second most ``difficult'' group (among the non-abelian groups of order less than $32$ that contain no cyclic subgroup of index two) was $H_{27}=UT_3(\mathbb{F}_3)$ for which our implementation of Algorithm~\ref{smalldavenportalgorithm} yielded the following data in about 40 seconds: $\mathsf{d}(H_{27})=6$ with $69026$ product-one free sequences partitioned by the similarity relation into $187$ equivalence classes. In case of the other groups the running times range from below one second up to half a minute.

Computation of the large Davenport constant using Algorithm \ref{largedavenportalgorithm} is more demanding. The most time-consuming group from the list was $SL_2(\mathbb{F}_3) = \tilde{A}_4$, with the large Davenport constant being 13, housing $499695$ atoms partitioned in $21033$ equivalence classes. In the case of this group the computation took about $40$ minutes. Excepting $C_3 \rtimes Dih_8 = (C_6 \times C_2) \rtimes_{\gamma} C_2$ and $S_4$ (computation taking about 24 minutes for each of them), the running times hardly reached one minute. For example in the interesting case of $H_{27}$, the large Davenport constant $\mathsf{D}(H_{27})=8$ was computed in $19$ seconds. This group has $108827$ atoms grouped in only $340$ equivalence classes (this explaining the relative quick computation within this group).
\end{remark}

\end{document}